\documentclass[superscriptaddress,twocolumn,showkeys,preprintnumbers,amsmath,amssymb]{revtex4-2}
\makeatletter  
\newcommand{\xleftrightarrows}[2][]{\mathrel{%  
 \raise.40ex\hbox{$  
       \ext@arrow 3095\leftarrowfill@{\phantom{#1}}{#2}$}%  
 \setbox0=\hbox{$\ext@arrow 0359\rightarrowfill@{#1}{\phantom{#2}}$}%  
 \kern-\wd0 \lower.40ex\box0}}

\def\leftrightarrowfill@{%  
 \arrowfill@\leftarrow\relbar\rightarrow%  
}  
\makeatother

\input{apalike}
\usepackage{amssymb,amsmath,amsthm}  
\usepackage[latin1]{inputenc}
\usepackage{color}
\usepackage{graphicx}
\usepackage{times}
\usepackage{dcolumn}% Align table columns on decimal point
\usepackage{bm}% bold math
\usepackage{subfigure}
\usepackage[bottom]{footmisc}

\newtheorem{teor}{Theorem} [section]

\newtheorem{prop}[teor]{Proposition}

\theoremstyle{definition}
\newtheorem{rem}[teor]{Remark}

\usepackage{pgfplots}%%per ficar tics a les figures
\tikzset{ font={\fontsize{8pt}{10}\selectfont}}%%size letter tics figures
\pgfplotsset{compat=1.16}

\usepackage{verbatim}

\newcommand{\eps}{\varepsilon}
\newcommand{\muu}{\mu}

\graphicspath{{figures/}}%%direccio per defecte de les figures

\begin{document}

\title{Dynamical mechanism behind ghosts unveiled in a map complexification}

\author{Jordi Canela}
\thanks{Corresponding author: J. Canela (canela@uji.es)}
\affiliation{\textit{Institut Universitari de Matem\`atiques i Aplicacions de Castell\'o, Universitat Jaume I, Castell\'o,  Spain}}
\author{Llu\'is Alsed\`a}
\affiliation{\textit{Departament de Matem\`atiques, Universitat Aut\`onoma de Barcelona, Cerdanyola del Vall\`es 08193,  Barcelona, Spain}}
\affiliation{\textit{Centre de Recerca Matem\`atica, Edifici C, Campus de Bellaterra. Cerdanyola del Vall\`es 08193,  Barcelona, Spain}}
\author{N\'uria Fagella}
\affiliation{\textit{Departament de Matem\`atiques i Inform\`atica, Universitat de Barcelona, Barcelona,  Spain}}
\author{Josep Sardany\'es} 
\thanks{Corresponding author: J. Sardany\'es (jsardanyes@crm.cat)}
\affiliation{\textit{Centre de Recerca Matem\`atica, Edifici C, Campus de Bellaterra. Cerdanyola del Vall\`es 08193,  Barcelona, Spain}}

%\\Phone:+34 935422834; Fax: +34 932213237}

\date{\today}


\begin{abstract}
Complex systems such as ecosystems, electronic circuits, lasers or chemical reactions can be modelled by dynamical systems which typically experience bifurcations. Transients typically suffer extremely long delays at the vicinity of bifurcations and it is also known that these transients follow scaling laws as the bifurcation parameter gets closer the bifurcation value in deterministic systems. The mechanisms involved in local bifurcations are well-known. However, for saddle-node bifurcations, the relevant dynamics after the bifurcation occur in the complex phase space. Hence, the mechanism responsible for the delays and the associated inverse-square root scaling law for this bifurcation can be better understood  by looking at the dynamics in the complex space. We follow this approach and  complexify a simple ecological system undergoing a saddle-node bifurcation. The discrete model describes a biological system with facilitation (cooperation) under habitat destruction for species with non-overlapping generations. We study the complex (as opposed to real) dynamics once the bifurcation has occurred. We identify the fundamental mechanism causing these long delays (called ghosts), given by two repellers in the complex space. Such repellers appear to be extremely close to the real line, thus forming a narrow channel close to the two new fixed points and responsible for the slow passage of the orbits, which remains tangible in the real numbers phase space. We analytically provide the relation between the inverse square-root scaling law and the multipliers of these repellers. We finally prove that the same phenomenon occurs for more general i.e., non-necessarily polynomial, models.  
\end{abstract}


\maketitle

\section{Introduction}

Bifurcations are responsible for qualitative changes in dynamical systems due to parameter changes~\cite{Kuznetsov1998,Strogatz2000}. Local bifurcations typically involve stability shifts or collisions between fixed points. Classical examples are transcritical, saddle-node (hereafter labeled as s-n, also named fold or tangent), pitchfork, or Hopf-Andronov bifurcations~\cite{Strogatz2000}. Bifurcations occur in most physical systems and have been mathematically described in
elastic-plastic materials \cite{Nielsen1993}, electronic circuits \cite{Kahan1999,Trickey1998}, 
or open quantum systems \cite{Ivanchenko2017}, among many others. Bifurcations have been also largely investigated in population dynamics ~\cite{Rietkerk2004,Staver2011,Carpenter2011,Hastings2018,Morozov2020} since they often involve important changes such as the separation between species' persistence and extinctions. Further theoretical research in socioecological systems \cite{May2008,Lade2013}, in autocatalytic systems~\cite{Fontich2008,Gimeno2018,Sardanyes2019}, %epidemic spreading, 
in the fixation of alleles in population genetics and biological or computer virus propagation~\cite{Murray2002,Ott2002,Hinrichsen2000,Odor2008}, has revealed bifurcation phenomena, often governed by abrupt changes (typically due to s-n bifurcations). Additionally, bifurcations have been identified experimentally in many physical~\cite{Gil1991,Trickey1998,Das2007,Gomez2017},
chemical~\cite{Maselko1982,Strizhak1996}, and biological systems \cite{Dai2012,Gu2014}. 

One of the most remarkable properties of systems approaching a local bifurcation is that transients' lengths slow down drastically. The length of these transients typically scales with the distance to the bifurcation value~\cite{Strogatz2000,Leonel2016}. Such scaling properties are found both in continuous-time (flows) and discrete-time (maps) dynamical systems. For instance, the length of transients, $\tau$, in transcritical bifurcations diverges as a power law  $\tau \sim |\mu - \mu_c|^{-1}$~\cite{Leonel2016,Teixeira2015}, with
$\mu$ and $\mu_c$ being, respectively, the control parameter and the value at which it bifurcates. The same scaling exponent is found in the supercritical Pitchfork bifurcation in flows~\cite{Leonel2016}, and in the Pitchfork and the period-doubling bifurcation in maps~\cite{Teixeira2015}. For the s-n, transients scale as $\tau \sim |\mu - \mu_c|^{-1/2}$ for flows and maps~\cite{Strogatz2000,Fontich2008,Duarte2011}. Remarkably, this scaling law was found experimentally in an electronic circuit modelling Duffing's oscillator~\cite{Trickey1998}. The same scaling exponent has been recently found in a delayed differential equation suffering a s-n bifurcation~\cite{Gimeno2018}. 

In short, the s-n bifurcation involves the collision and consequent annihilation of fixed points. This annihilation actually involves the jump of these two equilibria from the real numbers phase space to the complex  one~\cite{Strogatz2000,Sardanyes2006,Fontich2008,Sardanyes2019,Gimeno2018}. An interesting phenomenon tied to this fact is that, despite no fixed points are present in the real numbers phase space after the bifurcation takes place, the dynamics are still influenced by such points which now lie in the complex plane (for one-variable dynamical systems). This is why the orbits' delay right after the bifurcation is said to be governed by a ghost~\cite{Strogatz2000}. Ghost transients tied to s-n bifurcations have been identified in models of charge density waves~\cite{Strogatz1989}, hypercycles~\cite{Sardanyes2006,Sardanyes2007}, and ecological systems with facilitation: semi-arid ecosystems~\cite{Vidiella2018} or metapopulations with facilitation and habitat destruction~\cite{Sardanyes2019}. Notwithstanding, the interest in grasping the fundamental mechanism behind the ghost effect has provided few works deriving the inverse square-root scaling law using complex variable~\cite{Fontich2008,DJMS}. Despite these investigations, the dynamical mechanism taking place at the complex phase space and responsible for such delays still remains unknown.

It is well known that complex variables and complex analysis techniques are useful in areas of physics. For instance, in thermodynamics~\cite{Szostakiewicz2014}, hydrodynamics~\cite{Coleman1984,Marner2017} and quantum mechanics~\cite{Dirac1937}. By extension, the use of complex analysis also has applications in engineering fields such as nuclear~\cite{Cacuci2010}, aerospace~\cite{Poozesh2016}, or electrical~\cite{Bird2007} engineering, among many others; or in many branches of mathematics. Particularly, in dynamical systems, one can view not only the real phase space but also the  parameter space of analytic maps as restrictions of their complex analogues. This complexification has proven to be fruitful towards the understanding of phenomena often hidden from the real phase space. For instance,  complex techniques are needed to obtain an upper bound on the number of stable equilibria in an analytic real map, in terms of the number of zeros of its derivative. A  paradigmatic example is given by the quadratic family $x^2+c$ (or equivalently the logistic one $\lambda x(1-x)$) which,  via renormalization, models the local dynamics of every analytic map near a simple critical point. Complexifying both the variable $x$ and the parameter $c$, we are led to the well known Mandelbrot set, which allows a complete understanding of the cascade of period doubling bifurcations observed in the real parameter space. In this same family, the density of hyperbolic parameters in the real parameter line is also a recent fundamental result which required sophisticated complex techniques \cite{GS,Lyu}. 

In this paper we investigate the complex (as opposed to real)  dynamics after a s-n bifurcation in a map describing the dynamics of species facilitation under habitat destruction considering non-overlapping generations (see Ref.~\cite{Sardanyes2019} for a continuous-time approach). We have found that the long delays after the s-n bifurcation are due to two repellers symmetrically located in the complex plane above and below the real line. Short after the s-n bifurcation value, these fixed points are extremely close to the real line, leaving a narrow channel by which real orbits must go through, with extremely slow speed. We derive the scaling law using complex techniques, obtaining a simple relation between the multipliers of the repellers and the inverse square-root scaling law. Finally, we extend these results to more general (non-necessarily polynomial) families which exhibit s-n bifurcations.

\section{Mathematical model and results} 
%\subsection{Single-species dynamics with facilitation and habitat destruction}
The model we analyse in this paper describes the population dynamics of a single-species where the individuals cooperate through facilitation processes, in the context of  metapopulations with facilitation and habitat destruction. Specifically, given $x_n\geq 0$ the population at stage $n\geq 0$, the map is given by  
\begin{equation}
F(x_n) = x_{n+1} = x_n  + \muu x^2_n ( 1 - D - x_n) - \gamma x_n.
\label{map}
\end{equation}
This is a discrete-time version of the model  $\dot{x} = \muu x^2 \left( 1 - D - x\right) - \gamma x$ studied in ~\cite{Sardanyes2019}, 
%
%The model we analyse in this paper was recently studied in the context of metapopulations with facilitation and habitat destruction in a time-continuous setting~\cite{Sardanyes2019}. Specifically, we analyse an equivalent map, which  describes the population dynamics of a single-species where the individuals cooperate through facilitation processes. The map, given by
%\begin{equation}
%F(x_n) = x_{n+1} = x_n  + \muu x^2_n ( 1 - D - x_n) - \gamma x_n,
%\label{map}
%\end{equation}
%is a discrete-time version of the model  $\dot{x} = \muu x^2 \left( 1 - D - x\right) - \gamma x$ (see Ref.~\cite{Sardanyes2019}), 
%
obtained using an Euler step (i.e. assuming $x_{n+1} - x_n \approx dx/dt$). The constant $\mu >0$ is the intrinsic growth rate while $\gamma$ denotes the density-independent death rate of individuals. Note that the population includes a logistic-like growth constrain introducing intra-specific competition with a normalised carrying capacity. This logistic function also includes a fraction of habitat destroyed, $D \in [0,1]$. For the purposes of our work we will limit the range to $\gamma \in (0,1]$, mainly focusing on the impact of parameter $D$ in the delaying effects right after the saddle-node (hereafter s-n) bifurcation.

Most of the previous research on ghost transients has focused on time-continuous systems. These mainly include autocatalytic replicators~\cite{Sardanyes2006,Sardanyes2007,Fontich2008} and metapopulations~\cite{Fontich2010}. In order to study such delays in populations with non-overlapping generations such as insects or annual plants, it is better to use a discrete-time approach, as the one given by Eq.~\eqref{map}, although this is rarely found in the literature.  Nevertheless, some works explored s-n bifurcations in maps within the framework of the so-called Allee effects~\cite{Duarte2012} and single-species ecological models with harvesting~\cite{DJMS}. Also, spatial dynamics for hypercycles (two-species cross-catalytic systems) revealed the presence of ghost transients~\cite{Sardanyes2006a} as well.

%The process of facilitation is introduced by the growth term $\mu x^2$. This growth term obeys the hyperbolic growth associated to systems with autocatalytic feedbacks~\cite{Sardanyes2006,Sardanyes2007}.

%Equation~\eqref{edo} has three equilibrium points $P_0^* = 0$, and  $$\phantom{xx} P_{\pm}^* = \frac{1}{2} \left( 1 - D \pm \sqrt{(1-D)^2 - \frac{4 \gamma}{\mu}} \right).$$ The points $P_{\pm}$ suffer a saddle-node bifurcation when their discriminant equals zero. From this condition, we can obtain the bifurcation value, which depends upon all the parameters. This bifurcation value arises from a quadratic equation for $D$ of the discriminant, but since $D \in [0,1]$ we take:
%\begin{equation}
%D_c = 1 - 2\sqrt{\gamma/ \mu}.
%\end{equation}
 %As we will see below, the fixed points and the stability of Map~\eqref{map} are in agreement with the continuous time model. 

\subsection{Dynamics on the reals} \label{sec:fixos}

The fixed points of Map~\eqref{map} are computed from $F(x) = x$, obtaining $x_0^* =  0$, and the pair 
$$x_{\pm}^*  = \frac{1}{2} \left( 1 - D \pm \sqrt{(1-D)^2 - \frac{4 \gamma}{\mu}} \right).$$ 
The system undergoes  a s-n bifurcation when the discriminant of $x_{\pm}^*$ equals zero.
This gives bifurcation values for every one of the parameters involved:
\[
\mu_c = \frac{4\gamma}{(1-D)^2}, \phantom{xx} \gamma_c = \frac{\mu (1-D)^2}{4},  \phantom{xx} D_c = 1 - 2 \sqrt{\gamma/\mu},
\]
although we will use mostly $D\in[0,1]$ (fraction of habitat destroyed) as the control parameter.

We now focus on the properties of map~\eqref{map} for parameter values close to the s-n bifurcation. Let us start with the linear stability of the fixed points. Given a fixed point $x_0$ of a differentiable map $F(x)$, its multiplier is given by $\lambda(x_0)=F'(x_0)$. The fixed point is called \emph{attracting} if $|\lambda(x_0)|<1$, \emph{repelling} if $|\lambda(x_0)|>1$, and \emph{indifferent} if $|\lambda(x_0)|=1$, where $|\cdot|$ denotes the absolute value (or the complex modulus). If the  multiplier equals one we say that the fixed point is  \emph{parabolic} and, if the map is analytic,  it can be written locally as 
$
F(x)=x + c(x-x_0)^{\nu}+\mathcal{O}\left((x-x_0)^{\nu+1}\right),
$
 where $c\in\mathbb{R}\setminus\{0\}$ and $\nu\geq2$. In this case $x_0$ is said to be a fixed point of multiplicity $\nu$ since it is a solution of multiplicity $\nu$ of the equation $F(x)=x$.

The stability of the fixed points of  Map~\eqref{map}  can be obtained from the multiplier, which is computed as
$$\lambda(x) =F'(x)= \frac{dF(x)}{dx} =  1-\gamma+ 2 \muu x \left(1-D - \frac{3x}{2}\right).$$
The stability of $x_0^*$ is computed from $|\lambda(0)| =|1-\gamma|$. Note that the modulus of the multiplier is always less than $1$ for the considered range of $0 < \gamma \leq 1$. Hence, this fixed point is locally stable (it is attracting), being a super-attractor for $\gamma = 1$ since the multiplier equals 0.

Next, we study the stability of the fixed points $x_{\pm}^*$ near the parameter $D_c$. To do so, we fix $0 < \gamma \leq 1$ and $\mu>0$ and denote $\eps:=D-D_c$ (equivalently, $D=D_c+\eps$). We denote by $F_{\eps}$ the map $F$ with parameters $\gamma$, $\mu$, and $D_{c}+\eps$. We also denote by  $x_c=(1-D_c)/2$ the parabolic fixed point at which the fixed points  $x_{\pm}^*$ collide when $\eps=0$. Using this notation, $F_{\eps}$ is expressed as
\begin{equation}
F_{\eps}(x) = (1-\gamma)x  + \muu  ( 2x_c -\eps) x^2 - \muu x^3.
\label{eq:Feps}
\end{equation}

To point out the dependence of $x_{\pm}^*$ on $\eps$, we denote the fixed points by $x_{\pm}^*(\eps)$. A direct study of the stability of $x_{\pm}^*(\eps)$ in terms of $\eps$ yields complicated expressions.
 To sort out this problem we consider the map $G_{\eps}(y)=F_{\eps}(y+x_c)-x_c$, obtained by conjugating $F_{\eps}$ with the translation $x\rightarrow x-x_c$. Thus the dynamics of $G_{\eps}$ coincides with the dynamics of $F_{\eps}$ shifted by translation.  In particular, $y=-x_c$ is the attracting fixed point of $G_{\eps}$ which corresponds to the attracting fixed point $x=0$ of $F_{\eps}$, while bifurcation occurs at the double fixed point $y=0$. For $\eps\neq 0$, the map $G_{\eps}(y)$ has two fixed points around $0$, given by $y_{\pm}(\eps)=x_{\pm}^*(\eps)-x_c$. Moreover, the multipliers of $y_{\pm}(\eps)$, $\lambda(y_{\pm}(\eps))=G'(y_{\pm}(\eps))$, coincide with the multipliers of $x_{\pm}^*(\eps)$, $\lambda(x_{\pm}^*(\eps))=F'(x_{\pm}^*(\eps))$. Therefore, in order to study the stability of $x_{\pm}^*(\eps)$ it is enough to study the  stability of $y_{\pm}(\eps)$, task which occupies the remaining part of this section.  In Fig.~\ref{fig:epsesquema} we show a summary of the stability of $ x_{\pm}^*(\eps)$ for $|\eps|$ small.

\begin{figure}%[hbt]
{\centering 
  \includegraphics[width=0.35\textwidth]{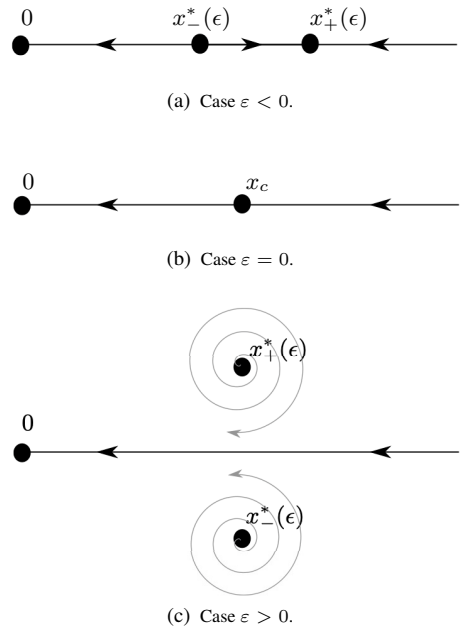}}
\caption{Scheme of the saddle-node bifurcation which takes place around $x_c$ for $|\eps|$ small. (a) Before bifurcation three fixed points are found in the real line. (b) At bifurcation both fixed points $x^*_-(\eps)$ and $x^*_+(\eps)$ collide at $x_c$. (c) After the bifurcation, $\eps>0$, both fixed points are placed in the complex plane.}
\label{fig:epsesquema}
\end{figure}

The map $G_{\eps}$ is given by
\begin{equation}
G_{\eps}(y) = - \mu x_c^2\eps+(1-2\mu x_c\eps)y-\mu(x_c+\eps)y^2-\mu y^3.
\label{eq:Geps}
\end{equation}
The fixed points $y_{\pm}(\eps)=x_{\pm}^*(\eps)-x_c$ are given by
\begin{equation} \label{eq:fixedG}
y_{\pm}(\eps)=-\frac{\eps}{2}\pm\sqrt{-x_c\eps +\frac{\eps^2}{4}}.
\end{equation}

If  $\eps<0$, $|\eps|$ small, the fixed points $y_{\pm}(\eps)$ can be written as
$$
 y_{\pm}(\eps)=\pm \, x_c^{1/2}\sqrt{-\eps}+\mathcal{O}(\eps).
$$
Notice that since $D_c\in[0,1]$ (by assumption $D\in[0,1]$), the number $x_c=(1-D_c)/2$ is positive. Therefore, the points $y_{\pm}(\eps)$ are real. Their multipliers, $\lambda(y_{\pm}(\eps))=G_{\eps}'(y_{\pm}(\eps))$,  are given by
$$\lambda(y_{\pm}(\eps))=1 \mp 2\, \mu \, x_c^{3/2}\sqrt{-\eps}+\mathcal{O}(\eps)=\lambda(x^*_{\pm}(\eps)).$$
Hence, $y_-(\eps)$ is repelling and $y_+(\eps)$ is attracting. We conclude that for $\eps\lesssim 0$,  the fixed points $x_{\pm}^*(\eps)$ of the original system $F_{\eps}$ satisfy $0<x_-^*(\eps)<x_+^*(\eps)$, $x_-^*(\eps)$ is repelling, and $x_+^*(\eps)$ is attracting (see Fig.~\ref{fig:epsesquema} (a)).

On the other hand, if $\eps>0$, $\eps $ small, then the fixed points are complex numbers given by
\[
 y_{\pm}(\eps)=\pm \,i \; x_c^{1/2}\eps^{1/2}+\mathcal{O}(\eps),
 \]
where $i=\sqrt{-1}$. Their multipliers, $\lambda(y_{\pm}(\eps))=G_{\eps}'(y_{\pm}(\eps))$, are given by
\begin{equation}\label{eq:multy-}
\lambda(y_{\pm}(\eps))=1 \mp 2\, i\cdot  \mu \, x_c^{3/2}\eps^{1/2}+\mathcal{O}(\eps)=\lambda(x^*_{\pm}(\eps)).
\end{equation}
In particular, if $\eps$ is small then $|\lambda(y_{\pm}(\eps))|>1$. Therefore, the fixed points $y_{\pm}(\eps)$ are repelling, and so are the fixed points $x_\pm^*(\eps)$ of $F_{\eps}$ (see Fig.~\ref{fig:epsesquema} (c)).

\begin{figure*}%[hbt]
{\centering 
  \includegraphics[width=\textwidth]{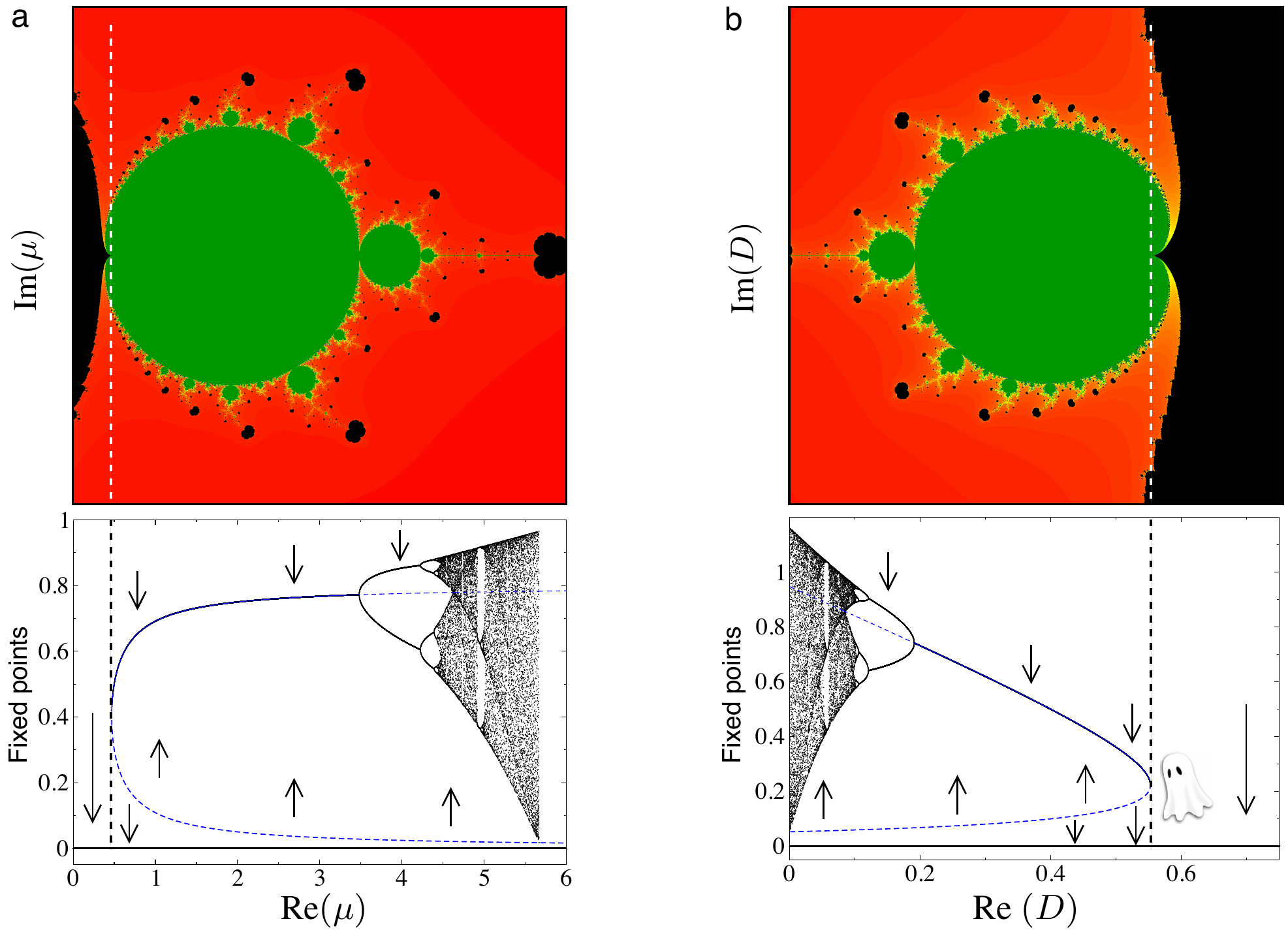}}
\caption{Bifurcation diagrams: (a) $D = 0.2$, $\gamma = 0.075$, for $0 \leq \mu \leq 6$. Here $\mu_c = 0.46875$; (b) $\mu = 4$, $\gamma = 0.2$, for $0 \leq D \leq 0.75$. Here $D_c = 0.5527864045\cdots$. The upper figures show the bifurcation from a complex point of view (compare \S \ref{sec:complex}). The lower figures show the bifurcations from a real point of view. For the later case, we have numerically built the bifurcation diagrams using two different initial conditions $x_0 = 10^{-4}$ and $x_0 = 0.5$. The  blue lines correspond to the stable (solid) and unstable (dashed) fixed points, being the lower branch the fixed point $P_-^*$ and the upper one the point $P_+^*$. Note that the origin is locally stable.}
\label{fig:param}
\end{figure*}

\subsection{The saddle-node bifurcation from the complex}\label{sec:complex}

\begin{figure*}[hbt!]
\centering
\subfigure[\small{ $D=D_c-0.1$} ]{
    \begin{tikzpicture}
    \begin{axis}[width=0.43\textwidth, axis equal image, scale only axis,  enlargelimits=false, axis on top, %xtick={-6,0,6}, ytick={-6,0,6} 
    ]
      \addplot graphics[xmin=-0.6,xmax=1,ymin=-0.4,ymax=0.4] {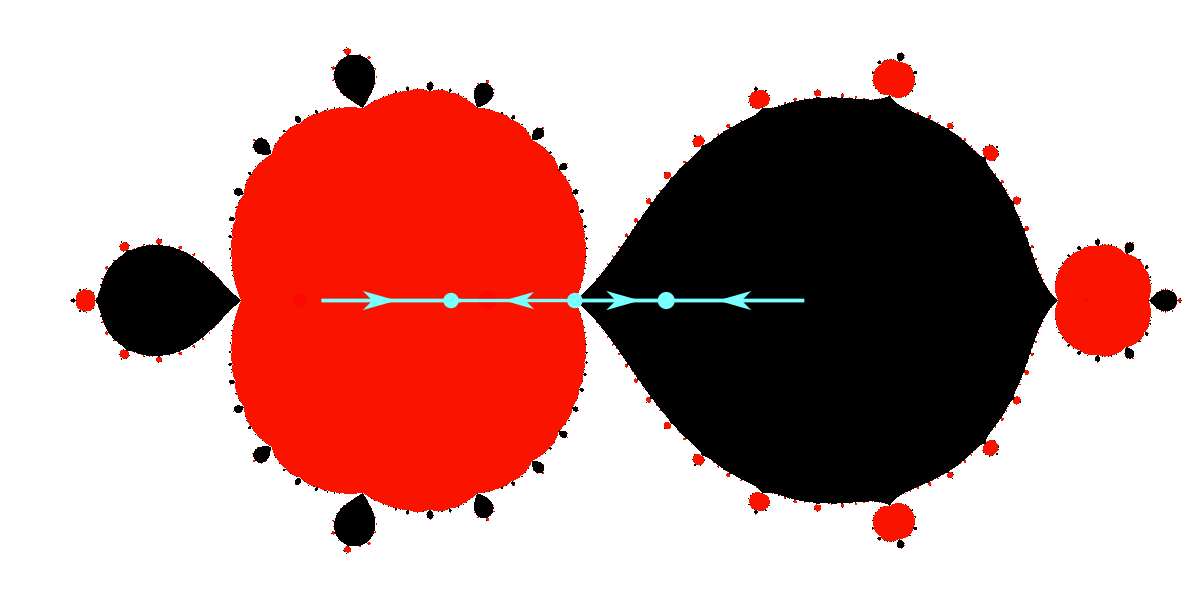};
    \end{axis}
  \end{tikzpicture}
\put(-145,70){ $0$}
\put(-128,70){ $x_-^*$}
\put(-108,70){ \textcolor{white}{$x_+^*$}}
  }
\subfigure[\small{ $D=D_c$} ]{
    \begin{tikzpicture}
    \begin{axis}[width=0.43\textwidth, axis equal image, scale only axis,  enlargelimits=false, axis on top, %xtick={-6,0,6}, ytick={-6,0,6} 
    ]
      \addplot graphics[xmin=-0.6,xmax=1,ymin=-0.4,ymax=0.4] {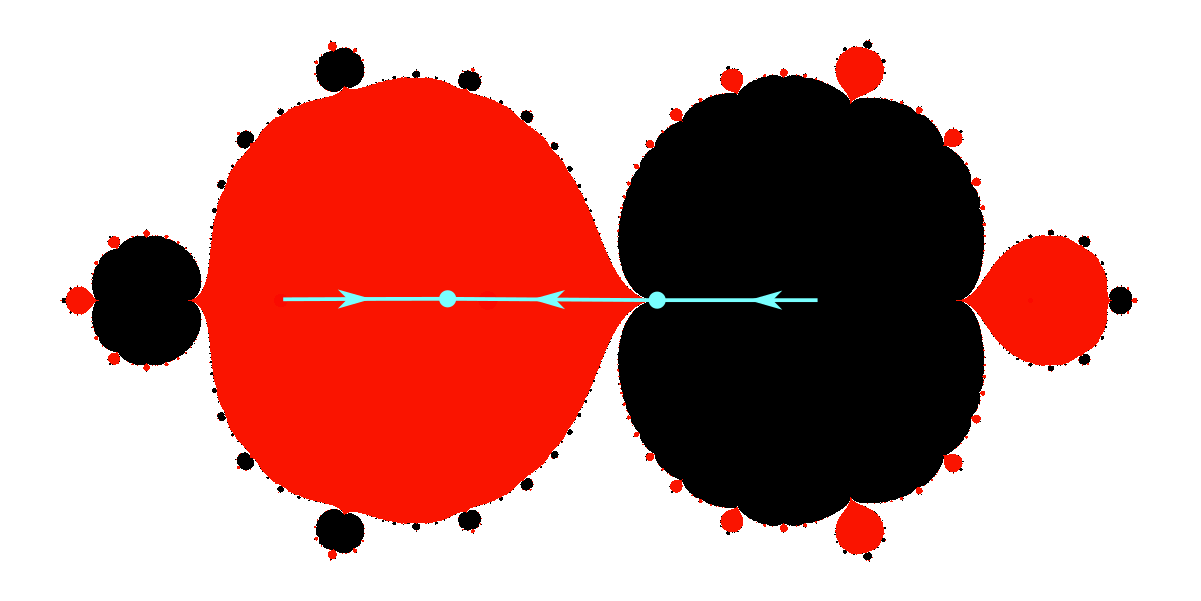};
    \end{axis}
  \end{tikzpicture}
\put(-145,70){ $0$}
\put(-108,70){ \textcolor{white}{$x_c$}}
  }
\subfigure[\small{ $D=D_c+10^{-6}$} ]{
    \begin{tikzpicture}
    \begin{axis}[width=0.43\textwidth, axis equal image, scale only axis,  enlargelimits=false, axis on top, %xtick={-6,0,6}, ytick={-6,0,6} 
    ]
      \addplot graphics[xmin=-0.6,xmax=1,ymin=-0.4,ymax=0.4] {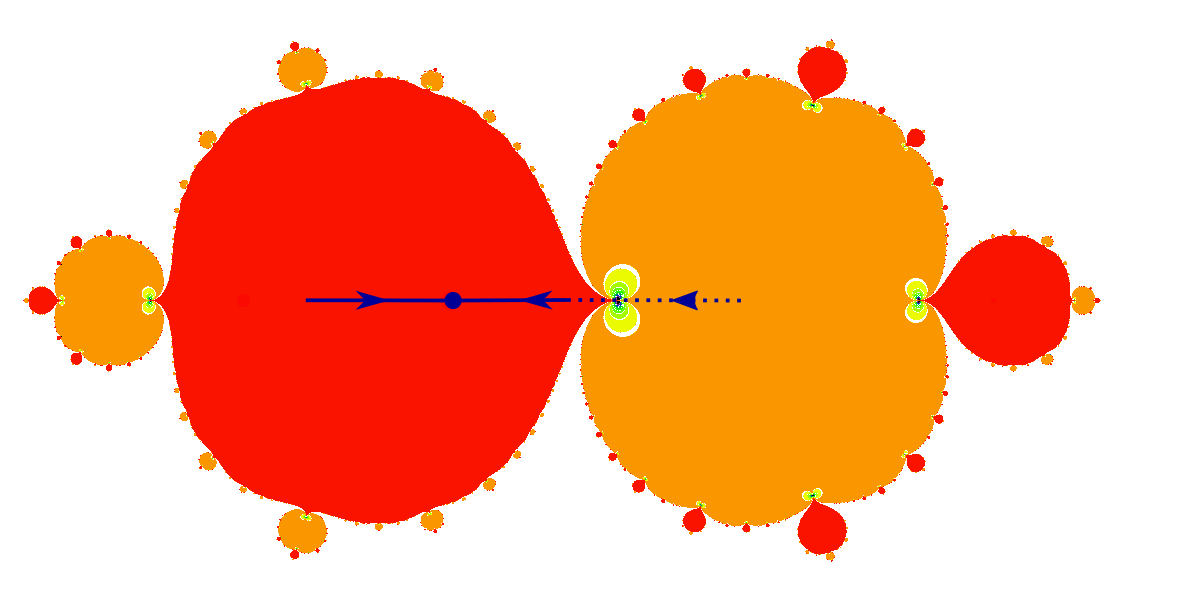};
    \end{axis}
  \end{tikzpicture}
\put(-145,70){ $0$}
  }
\subfigure[\small{ Zoom in on (c).} ]{
    \begin{tikzpicture}
    \begin{axis}[width=0.43\textwidth, axis equal image, scale only axis,  enlargelimits=false, axis on top, %xtick={-6,0,6}, ytick={-6,0,6} 
    ]
      \addplot graphics[xmin=0.2,xmax=0.25,ymin=-0.0125,ymax=0.0125] {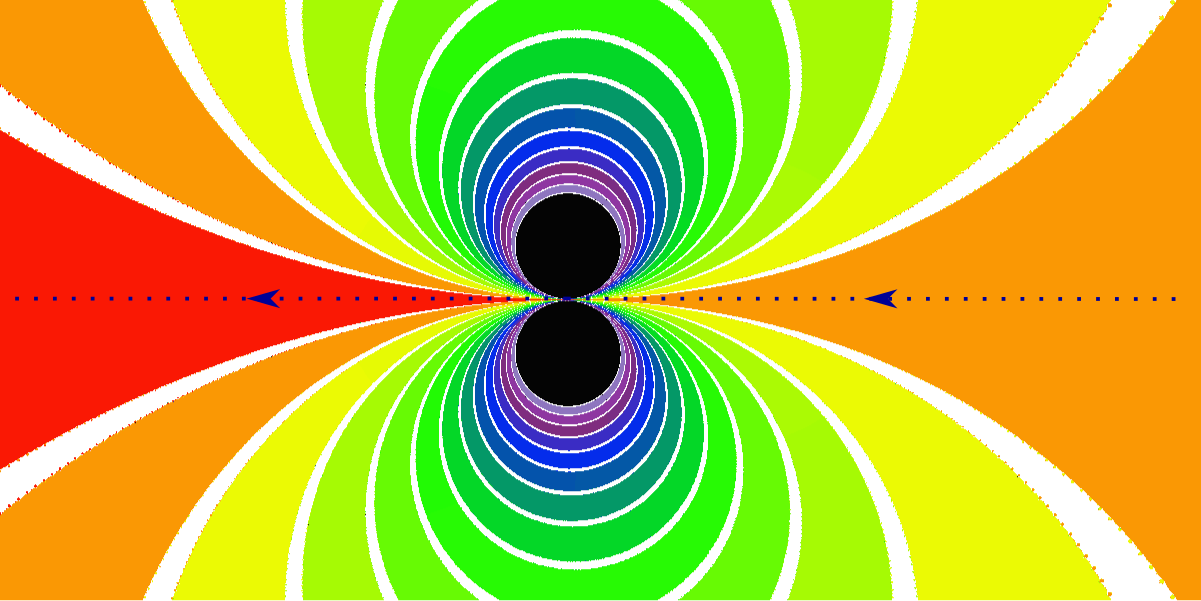};
    \end{axis}
  \end{tikzpicture}
\put(-132,68){ \textcolor{white}{$ \times$}}
\put(-132,75){ \textcolor{white}{$ x_+^*$}}
\put(-132,59){ \textcolor{white}{$ \times$}}
\put(-130,53){ \textcolor{white}{$ x_-^*$}}
  }
\caption{{\protect\small Dynamical planes of $F$ for  $\mu=4$, $\gamma=0.2$, and different values of $D$. White colour indicates that points converge to $z=\infty$, the scaling from purple (slow convergence) to red (fast convergence) indicates that the point converges to $z=0$, and black indicates that the point has not converged to a given neighbourhood of $z=0$ or $z=\infty$ after $10^5$ iterates. In figure (a) there is an attracting fixed point other than $z=0$. In figure (b) the former attracting fixed point becomes parabolic, attracting from the right and repelling from the left in the reals. In figures (c) and (d) the bifurcation occured and there are now two repelling fixed points in $\mathbb{C}\setminus\mathbb{R}$. The real points which converged to the parabolic point at bifurcation value, now converge slowly  to $x=0$. }}
\label{fig:dynam}
\end{figure*}

As discussed in the previous section, for $\eps\gtrsim 0$, the fixed points $x^*_{\pm}(\eps)$ become complex numbers, thus a dynamical study of $F_{\eps}$ from a complex point of view can help us to understand its dynamics.  Let us briefly recall a few concepts  of the theory of dynamics in one complex variable for polynomials (see \cite{Mi} for a  detailed introduction). As for real maps, a fixed  point $z_0$ of  a polynomial $f:\mathbb{C} \to \mathbb{C}$ is \emph{attracting, repelling} or {\it indifferent} depending on whether its multiplier $\lambda(z_0):=f'(z_0)$  is respectively smaller, larger or equal to one in modulus. Indifferent fixed points with multiplier $\lambda(z_0)=e^{2\pi i p/q}$ with $p/q\in\mathbb{Q}$ are also called parabolic. Likewise, we can classify periodic points substituting $f$ by $f^p$ where $p$ is the period of the orbit. 
%

%Let $\widehat{\mathbb{C}}$ denote the Riemann Sphere and let $f:\widehat{\mathbb{C}}\rightarrow\widehat{\mathbb{C}}$ be a polynomial. We consider the discrete dynamical system given by the iterations of $f$.
% We say that a point $z_0$ is periodic, of period $p$,  if $f^p(z_0)=z_0$ and $f^q(z_0)\neq z_0$ for $0< q<p$. The multiplier of a periodic point of period $p$ is given by $\lambda(z_0)=(f^p)'(z_0)=f'(z_0)\cdot f'(f(z_0))\cdot...\cdot f'(f^{p-1}(z_0))$. As it is the case for fixed points,  a periodic point $z_0$ is called \emph{attracting} if $|\lambda(z_0)|<1$, \emph{repelling} if $|\lambda(z_0)|>1$, and \emph{indifferent} if $|\lambda(z_0)|=1$, where $|\cdot|$ denotes the complex modulus. An indifferent point is called \emph{rationally indifferent} or \emph{parabolic} if $\lambda(z_0)=e^{2\pi ip/q}$, where $p/q\in\mathbb{Q}$. 
 
The {\em basin of attraction} $\mathcal{A}(z_0)$ of an attracting or parabolic fixed point $z_0$ is an open set of the complex plane consisiting, as usual, on the points whose orbits converge to $z_0$, which belongs to the interior of $\mathcal{A}(z_0)$ in the attracting case, and which lies on the boundary if it is parabolic.
 %
% nGiven an attracting or parabolic fixed  point $z_0$, there is a basin of attraction  $\mathcal{A}(z_0)$ which consists of the open set of points which converge under iteration of $f$ onto $z_0$. If the attracting or parabolic point is periodic of period $p$  we can also associate a basin of attraction to it, given by the set of points which accumulate onto $\langle z_0\rangle=\{z_0, f(z_0),..., f^{p-1}(z_0)\}$ under iteration of $f$. We want to remark that if $z_0$ is attracting then $z_0$ belongs to the interior of $\mathcal{A}(z_0)$.
%  However, if $z_0$ is parabolic then it belongs to the boundary of $\mathcal{A}(z_0)$. 
 %
An important feature of the dynamics of complex polynomials is that $z=\infty$ is always a superattracting fixed point, i.e.\ $\lambda(\infty)=0$ (computed with the appropriate chart). Therefore, in the dynamical plane of a polynomial we can always find an open set of initial conditions, $A(\infty)$, whose orbits converge to $z=\infty$. This basin of attraction is always connected, since $\infty$ has no preimages in the complex plane.

The dynamics of the complex polynomial $f$ induces a partition of $\mathbb{C}$ into two completely invariant sets: The \emph{Fatou set} $\mathcal{F}(f)$ of points for which the family $\{f^n\}_n$ of iterates of $f$ is equicontinuous  in some neighbourhood of $z$; and its complement $\mathcal{J}(f)= \mathbb{C}\setminus\mathcal{F}(f)$, the  \emph{Julia set}. The Fatou set is open and consists of the points around which the dynamics is stable, while the Julia set is closed, often  fractal, and corresponds to the set of points whose orbits are chaotic. 
%The connected components of the Fatou set, called \emph{Fatou components}, are mapped among themselves under iteration of $f$ and can only be periodic or preperiodic (see \cite{Su}). Every periodic component either belongs to a basin of attraction of an attracting or parabolic periodic point  or is a simply connected rotation domain (Siegel disk).} 
Note that basins of attraction of attracting and parabolic periodic points belong to $\mathcal{F}(f)$ while parabolic points themselves belong to $\mathcal{J}(f)$. 

Critical points, i.e.\ points $c$ such that $f'(c)=0$,  play an important role in holomorphic dynamics, since every basin of attraction must contain at least one critical point. This property bounds the number of possible stable equilibria that may coexist, and allows us to draw the bifrurcation set in the parameter space by colouring each parameter value according to   the asymptotic behaviour of the critical orbits (i.e. the orbits of the critical points) for that given parameter.  In Fig.~\ref{fig:param} (b) we show the complex $D-$plane of our model, the family $F=F_D$ in Eq.~\eqref{map}, for different fixed real values of $\mu$ and $\gamma$. The polynomials $F$ have two  critical points. Since $z=0$ is an attracting fixed point, the orbit of at least one of the critical points converges to $z=0$. Hence we colour the parameter in black if both critical orbits converge to $z=0$;  we use a scaling from yellow (slow convergence) to red (fast convergence) if one of the critical orbits converges to $z=\infty$, and we plot the parameter in green if one of the critical orbits converges neither to $z=0$ nor to $z=\infty$. Using this procedure, we can see how the open segments of real parameters $D$ for which $F$ has a real attracting periodic point other than $z=0$  (see the bifurcation diagrams in Fig.\ \ref{fig:param}) become open domains of complex parameters $D$ for which $F$ has an attracting periodic point in $\mathbb{C}$. Different connected components of the interior of the green set correspond to different periods of the attracting periodic orbit, being the largest one for period 1.

\begin{figure*}%[hbt]
{\centering 
   \includegraphics[width=0.85\textwidth]{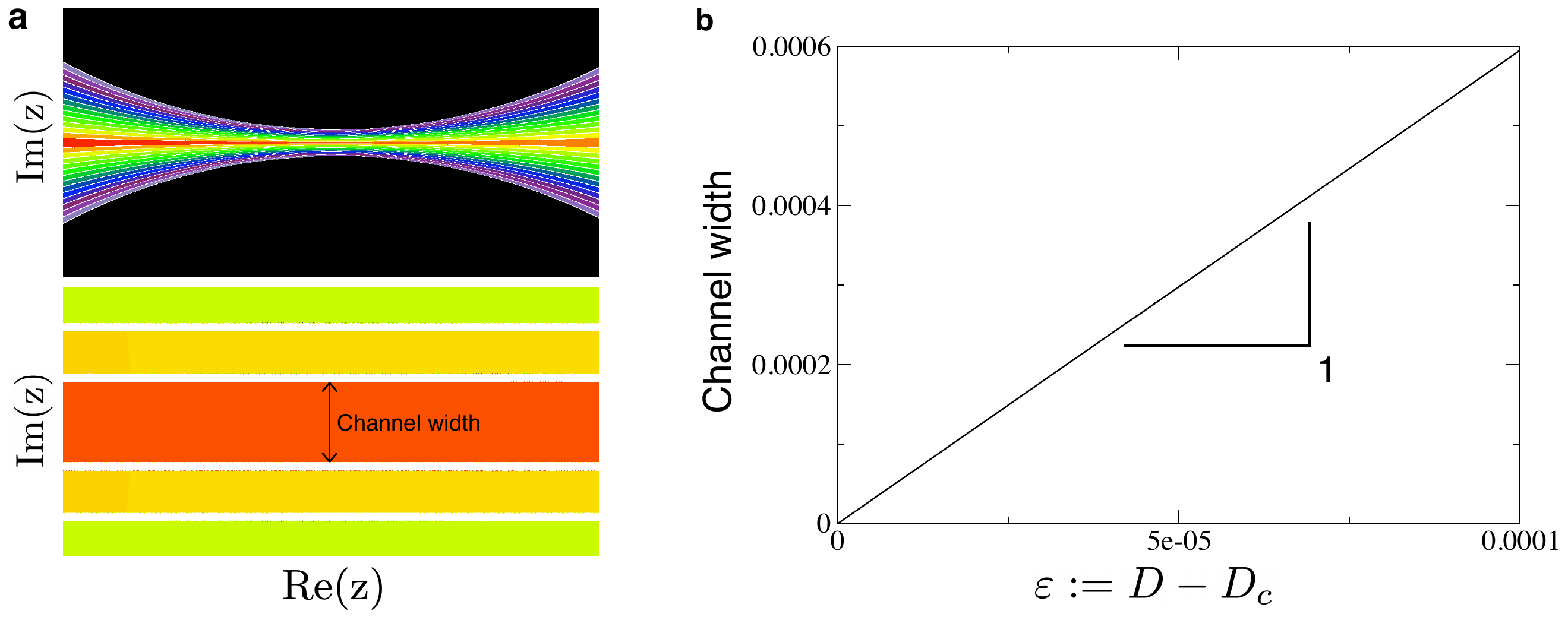}
   }
\caption{\small (a) Zooms of Fig.~\ref{fig:dynam} (d). The ranges of the initial conditions are $\rm{Re}(z)\in[x_c-10^{-4}, x_c+10^{-4}]$ and  $\rm{Im}(z)\in[-10^{-4}/2, 10^{-4}/2]$ (upper panel, labeled specular Dark Side Of The Moon) and $\rm{Re}(z)\in[x_c-2\cdot10^{-5}, x_c+2\cdot10^{-5}]$ and  $\rm{Im}(z)\in[-10^{-5}, 10^{-5}]$ (lower panel). (b) Dependence of the channel width on $\eps$ .}
\label{fig:canal}
\end{figure*}

We can now describe the s-n bifurcation from the complex plane. To that end we fix the parameters $\mu=4$ and $\gamma=0.2$ and draw the dynamical plane of $F_{\eps}(z)$, where $z\in\mathbb{C}$,  for values of $\eps$ before and after bifurcation (see Fig.~\ref{fig:dynam}). The pictures are done as follows. We take a grid of $1200\times 600$ points and iterate the corresponding initial conditions up to $10^5$ times. If the orbit reaches a given small neighbourhood of the attracting fixed point $z=0$ we plot the point using an scaling from red (fast convergence) to green, blue, purple and to grey (slow convergence). If the orbit escapes to $z=\infty$ we plot the point in white and, if after $10^5$ iterates the orbit has neither converged to $z=0$ nor to $z=\infty$ we plot the point in black. 

In Fig.~\ref{fig:dynam} (a) we plot the dynamical plane of $F_{\eps}$ for $\eps=-0.1$.  Since $\eps\lesssim 0$, the point $x_+^*(\eps)$ is attracting, and its basin of attraction can be seen in black. In red we see the basin of attraction of $z=0$.  The repelling fixed point $x_-^*(\eps)$ is the intersection point between the closure of the big red component, which contains $z=0$, and the closure of the big black component. 
In Fig.~\ref{fig:dynam} (b) we plot the dynamical plane of $F_{\eps}$ for $\eps=0$, the bifurcation parameter, for which $x_-^*(\eps)$ and $x_+^*(\eps)$ collide at the parabolic fixed point $x_c$, which is now the intersection point between the closure of the two basins. As before, red orbits converge to $z=0$ while black orbits now converge to to the parabolic point $x_c$. Notice that, since $x_c\in\mathcal{J}(F_{0})$, the dynamics around $x_c$ is not stable. 
Finally, in Fig.~\ref{fig:dynam} (c), we plot the dynamical plane of $F_{\eps}$ for $\eps=10^6$, for which we see how $x_-^*(\eps)$ and $x_+^*(\eps)$ exited the real line and became repelling - they are the center of the two spirals which appear near the former parabolic point in  Fig.~\ref{fig:dynam} (d), a zoom of (c). For this parameter we see with the scaling from red to green, blue, purple and grey the points which converge under iteration of $F_{\eps}$ to $z=0$. We observe how most of the points (in particular all the real ones) which converged to the parabolic point $x_c$ before perturbation now converge to $z=0$. However, we can see white orbits (i.e. orbits which converge to infinity) forming infinite spirals around $x_{\pm}^*(\eps)$, since now these two fixed points must belong to the boundary of the basin of infinity. 

All coloured orbits around $x_{\pm}^*(\eps)$ converge to $z=0$, but they need an increasing amount of iterates to do so, which is reflected in the colour. The two black disks, which correspond to initial conditions which have neither converged to $z=0$ nor escaped to $z=\infty$ in $10^5$ iterates, are numerical artifacts, and they would disappear if we increased  the number of iterates {\em ad infinitum}. These two spirals are, somehow, the complex dynamical obstruction which leads to the slow iteration of the real points through the former parabolic point $x_c$. Indeed, since the points $x_{\pm}^*(\eps)$ are very close, the channel between them is very narrow (see Fig.~\ref{fig:canal}), which leads to a derivative very close to one and hence to extremely slow dynamics (see Fig.~\ref{fig:scalinglaw}). Numerical experiments indicate that the channel width depends linearly in $\eps$ for $|\eps|$ small. For instance, for $\mu=4$ and $\gamma=0.2$ the channel width is approximately $5.9476 \cdot \eps$.  
 
  In the next section we analyse the relation between  the multipliers of the fixed points $x_{\pm}^*(\eps)$ and the time required to pass through the channel.

\subsection{Analytical derivation of the scaling-law using the multiplier}

For $D=D_c$ (i.e. $\eps=0$) the parabolic fixed point $x_c$ of $F_{0}$ is repelling from the left and attracting from the right (see Fig.~\ref{fig:epsesquema} (b), compare Fig.~\ref{fig:dynam} (b)). This follows from the fact that the second derivative of $F_{0}$ at $x_c$ is negative. Moreover, all points in the segment $(0, x_c)$ converge under iteration of $F_0$ to $x=0$. As discussed in the previous section, for $\eps\gtrsim 0$, the fixed points $x_{\pm}^*(\eps)$ become repelling complex fixed points. Moreover, the real points which converge to the parabolic point $x_c$ before perturbation, converge to $x=0$ after perturbation if $\eps$ is small enough (see Fig.~\ref{fig:dynam} (c), compare  Fig.~\ref{fig:epsesquema} (c)).

Let $x_{ini}>x_c$ be a point in the in the immediate basin of attraction of $x_c$ under $F_0$ (all points in the segment $(x_c, x_{ini}]$ converge under iteration of $F_0$ to $x_c$). The goal of this section is to determine the `asymptotic' number of iterates required by $x_{ini}$ to be mapped onto a given neighbourhood of the attracting fixed point $x=0$ under $F_{\eps}$. The next theorem estimates this quantity in terms of the multiplier  $\lambda(x_-^*(\eps))$ and recovers from it the scaling-law in terms of $\eps$. 

\begin{teor}\label{teor:scaling0}
	Let $x_{ini}>x_c$ be a point in the in the immediate basin of attraction of $x_c$ under $F_0$. Let $\eps>0$. Then, the number of iterates $N_{\eps}$ required to go from $x_{ini}$ to a given (fixed) neighbourhood of $x_0$ under $F_{\eps}$ is given by
	$$
	N_{\eps}\approx \frac{2\pi}{\rm{Im}(\lambda( x_-^*(\eps))}+K_1= \frac{\pi}{\mu\; x_c^{3/2}\;\eps^{1/2}}+K_2,
	$$
	where $K_1,K_2\in\mathbb{R}$ and $K_2=K_1+\mathcal{O}(1)$. 
\end{teor}
\proof

Let $\delta>0$ fixed. For the unperturbed map $F_0$, the point $x_{ini}$ is mapped to the left of $x_c+\delta$ in a finite number of iterates of $F_0$. Analogously, the point $x_c-\delta$ is mapped onto a given neighbourhood of $x=0$ in a finite numbers of iterates. It follows that, for $\eps$ small, we can bound the number of iterates employed to go from $x_{ini}$ to the left of $x_c+\delta$ and from $x_c-\delta$ to a given neighbourhood of $x=0$ under $F_{\eps}$ by a constant $K_0$. Therefore, in order to estimate the number of iterates used to go from $x_{ini}$ onto a neighbourhood of $x=0$ it is enough to estimate the number of iterates required to go from $x_c+\delta$ to $x_c-\delta$ under $F_{\eps}$.

If $\epsilon$ and $\delta$ are small enough, since both $F_\eps(z)-z$ and $F_\eps'(z)-1$ are very small in modulus, the discrete model can be approximated by the vector field $z'=F_\eps(z)$, and hence the number of iterates $N_{\eps, \delta}$ required to go from $x_c+\delta$ to $x_c-\delta$ can be approximated by the integral
$$\int_{x_c+\delta}^{x_c-\delta}\frac{dx}{F_{\eps}(x)-x},$$
(see \cite{DJMS} and \cite[page 4]{Oud}). As shown in \cite{Fontich2008} (see also \cite{DJMS}) this integral can be computed using complex analysis techniques. 
\begin{figure}[hbt!]
\centering
  \setlength{\unitlength}{220pt}%
 \begin{picture}(1,0.6)%
    \setlength\tabcolsep{0pt}%
    \put(0.1,0){\includegraphics[width=\unitlength,page=1]{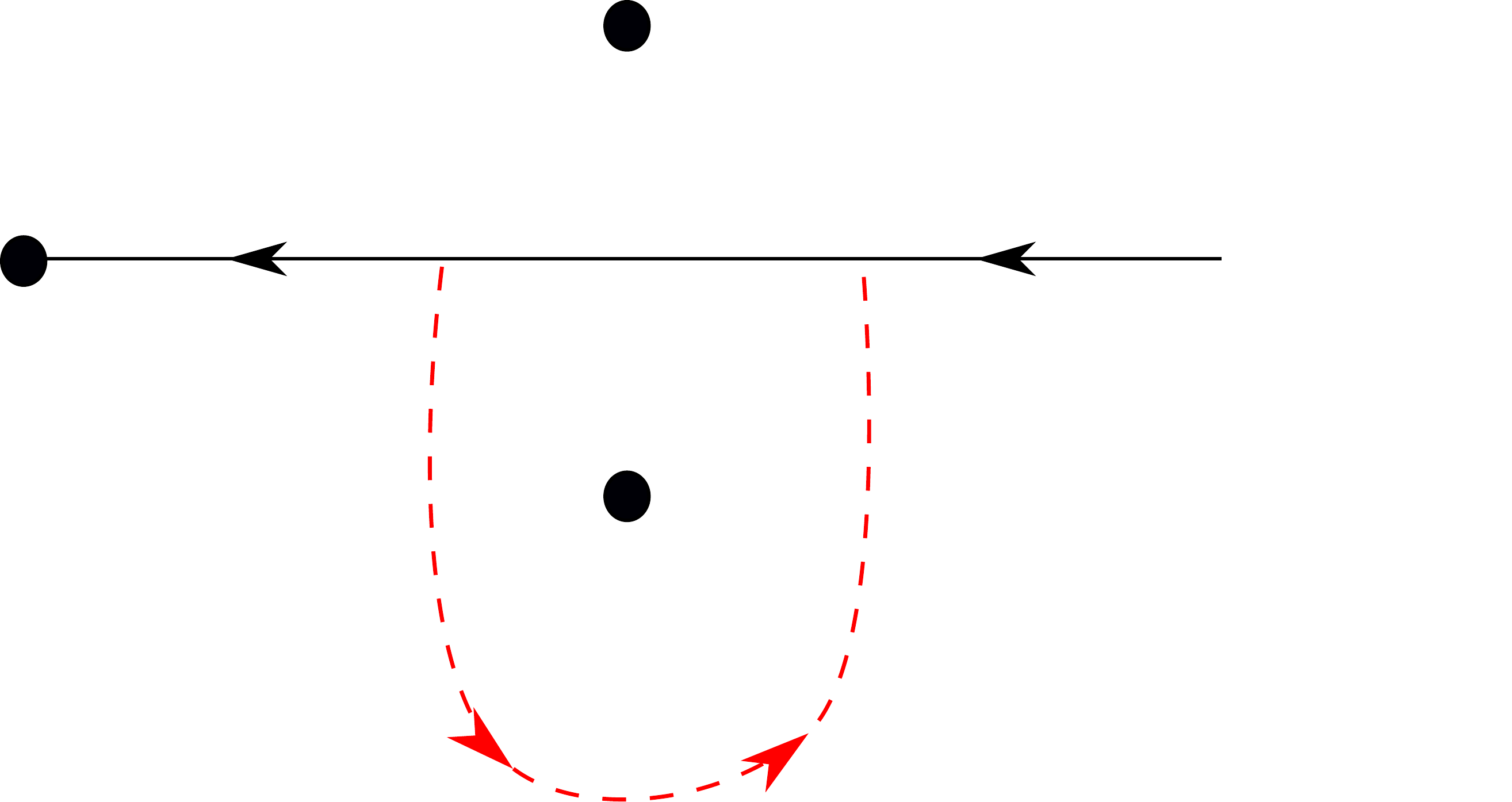}}%
    \put(0.51398567,0.22575472){\color[rgb]{0,0,0}\makebox(0,0)[lt]\smash{\begin{tabular}[t]{l}$x_-^*(\eps)$\end{tabular}}}%
    \put(0.11535925,0.39103877){\color[rgb]{0,0,0}\makebox(0,0)[lt]\smash{\begin{tabular}[t]{l}$0$\end{tabular}}}%
     \put(0.53343083,0.4688195){\color[rgb]{0,0,0}\makebox(0,0)[lt]\smash{\begin{tabular}[t]{l}$x_+^*(\eps)$\end{tabular}}}%
    \put(0.1,0){\includegraphics[width=\unitlength,page=2]{eta0.pdf}}%
    \put(0.3448125,0.393){\color[rgb]{0,0,0}\makebox(0,0)[lt]\smash{\begin{tabular}[t]{l}$x_c-\delta$\end{tabular}}}%
    \put(0.62676829,0.393){\color[rgb]{0,0,0}\makebox(0,0)[lt]\smash{\begin{tabular}[t]{l}$x_c+\delta$\end{tabular}}}%
    \put(0.44592754,0.04102524){\color[rgb]{1,0,0}\makebox(0,0)[lt]\smash{\begin{tabular}[t]{l}$\eta$\end{tabular}}}%
    \put(0.9,0.39){\color[rgb]{0,0,0}\makebox(0,0)[lt]\smash{\begin{tabular}[t]{l}$x_{ini}$\end{tabular}}}%
 \end{picture}%
\caption{\small Scheme of the curve $\eta$ used to estimate the number of iterates between $x_c+\delta$ and $x_c-\delta$.}
\label{fig:eta}
\end{figure}
Their idea is to consider a simple closed curve $\eta$, oriented counter-clockwise,  which intersects the real line exactly in the segment $[x_c+\delta, x_c-\delta]$ and surrounds the fixed point $x_-^*(\eps)$ (see Fig.~\ref{fig:eta}). Using the Residue Theorem, the integral over the whole curve $\eta$ is given by
\[
\oint_{\eta}\frac{dz}{F_{\eps}(z)-z}=2\pi i \;\mbox{Res}\left(\frac{1}{F_{\eps}(z)-z}, x_-^*(\eps)\right).
\]
Let $\eta':=\eta\setminus [x_c+\delta, x_c-\delta]$ and let $I_{\eps}^{\delta}:=\int_{\eta'}\frac{dz}{F_{\eps}(z)-z}$. Since for $\eps\gtrsim0$   the fixed points of $F_{\eps}(z)-z$ stay bounded away from the curve $\eta'$,  there exists $K_{\delta}>0$ independent from $\eps$ such that $\left| I_{\eps}^{\delta}\right|< K_{\delta}$. In fact, it is easy to see that $I_{\eps}^{\delta}=I_{0}^{\delta}+\mathcal{O}(\eps)$. Hence, the number of iterates $N_{\eps, \delta}$ required to go from $x_c+\delta$ to $x_c-\delta$ can be approximated by
{ \everymath={\displaystyle}
\begin{eqnarray*}
% \arraycolsep=1.4pt\def\arraystretch{2.2}
% \begin{equation}
N_{\eps, \delta} &\approx& \int_{x_c+\delta}^{x_c-\delta}\frac{dx}{F_{\eps}(x)-x} \\ &=&
 2\pi i \;\mbox{Res}\left(\frac{1}{F_{\eps}(z)-z}, x_-^*(\eps)\right)- I_{\eps}^{\delta}. \nonumber
\end{eqnarray*}
 Therefore, to understand the asymptotic behaviour of $N_{\eps, \delta}$ as $\eps$ tends to 0 it is enough to study $2\pi i \;\mbox{Res}\left(\frac{1}{F_{\eps}(z)-z}\right)$.

In this paper we replace the use of $2\pi i \;\mbox{Res}\left(\frac{1}{F_{\eps}(z)-z}\right)$ by the use of the holomorphic index of $F_{\eps}$ at $x_-^*(\eps)$. For a detailed introduction to holomorphic index we refer to \cite[\S12]{Mi}. Given a fixed point $z^*$ of a holomorphic map $g$, its holomorphic index is given by
$$\iota(g, z^*)=\frac{1}{2\pi i}\oint\frac{dz}{z-g(z)},$$
where the circular integral is done over any simple closed which does not contain any fixed point of $g$ and only surrounds the fixed point $z^*$. An important property of the holomorphic index is that if the multiplier $\lambda(z^*)$ of $z^*$ is different from 1, then
$$\iota(g, z^*)=\frac{1}{1-\lambda(z^*)}.$$
In particular,
$$\iota(F_{\eps}, x_-^*(\eps))=\frac{1}{2\pi i}\oint_{\eta}\frac{dz}{z-F_{\eps}(z)}=\frac{1}{1-\lambda( x_-^*(\eps))}.$$
Notice that both the $\mbox{Res}\left(\frac{1}{F_{\eps}(z)-z}, x_-^*(\eps)\right)$ and $\iota(F_{\eps}, x_-^*(\eps))$ provide the same information. Indeed,
\begin{widetext}
{ \everymath={\displaystyle}
\begin{equation}
 \arraycolsep=1.4pt\def\arraystretch{2.2}
% \begin{equation}
\oint_{\eta}\frac{dz}{F_{\eps}(z)-z} =2\pi i \;\mbox{Res}\left(\frac{1}{F_{\eps}(z)-z}, x_-^*(\eps)\right)
=-2\pi i \; \iota(F_{\eps}, x_-^*(\eps))=\frac{2\pi i}{\lambda( x_-^*(\eps))-1}.
 \label{eq:holind}
\end{equation}
}
\end{widetext}
Computing $\oint_{\eta}\frac{dz}{F_{\eps}(z)-z}$ using the holomorphic index instead of the explicit expression of the residue has two main advantages: It leads to simpler expressions and it brings up the multiplier of the bifurcated fixed points.  In particular, it allows us to relate the sought number of iterations with the complex dynamics which appear after the s-n bifurcation (see \S \ref{sec:complex}).

For the family $F_{\eps}$, the fixed point $x_-^*(\eps)$ has multiplier $\lambda(x_-^*(\eps))=\lambda(y_-(\eps))$  (compare \eqref{eq:multy-}). It follows that 
\begin{eqnarray*}
N_{\eps, \delta}&\approx& \int_{x_c+\delta}^{x_c-\delta}\frac{dx}{F_{\eps}(x)-x}
 = \frac{2\pi i}{\lambda( x_-^*(\eps))-1}- I_{\eps}^{\delta} \\&=& 
 \frac{2\pi}{\rm{Im}(\lambda( x_-^*(\eps))}+\mathcal{O}(1)- I_{\eps}^{\delta}.
\end{eqnarray*}

  We want to remark that the numbers $\frac{2\pi i}{\lambda( x_-^*(\eps))-1}$ and $I_{\delta}$ are in general complex. However, since $\int_{x_c+\delta}^{x_c-\delta}\frac{dx}{F_{\eps}(x)-x}$ is real, their difference is real. Therefore, the quantity $\mathcal{O}(1)-I_{\eps}^{\delta}$ is a real number. We also want to remark that from the previous computation we can recover the scaling law in terms of $\eps$. Indeed, 

\[
 \frac{2\pi}{\rm{Im}(\lambda( x_-^*(\eps))}=\frac{\pi}{\mu\; x_c^{3/2}\;\eps^{1/2}}+\mathcal{O}(1).
 \]

We can conclude that  the number of iterates $N_{\eps}$ required by the point $x_{ini}$ to be mapped onto a given neighbourhood of $x=0$ under $F_{\eps}$ grows like
$$
N_{\eps}\approx \frac{2\pi}{\rm{Im}(\lambda( x_-^*(\eps))}+K_1= \frac{\pi}{\mu\; x_c^{3/2}\;\eps^{1/2}}+K_2,
$$
where $K_1,K_2\in\mathbb{R}$.
 \endproof

In Fig.~\ref{fig:scalinglaw} we provide numerical experiments showing how this scaling-law of the number of iterates is satisfied for the family $F_{\eps}$.

\begin{figure*}%[hbt]
{\centering 
   \includegraphics[width=0.96\textwidth]{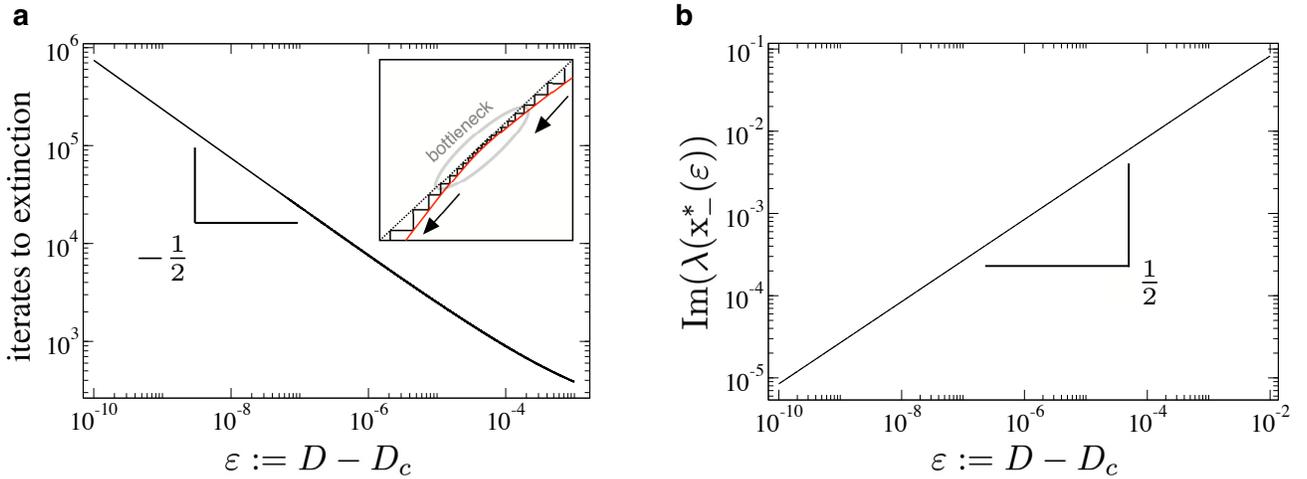}
   }
\caption{After the bifurcation the two fixed points that collided move to the complex  phase space becoming repulsors with complex derivative and hence spiriling nearby orbits. The horizontal axis displays the value of $\eps$ while the vertical axis records the number of iterates to cross the bottleneck (Fig.\ a) and the imaginary part of $\lambda( x_-^*(\eps))$ (Fig.\ b), right after the bifurcation with $\mu=4$ and $\gamma =2$. The inset displays an enlarged view of the  cobweb map with the bottleneck region generated once the map crosses the diagonal (right after the s-n bifurcation). Note the a lot of iterates are needed to cross the bottleneck (here we have also used $\mu=4$ and $\gamma =2$ and $\eps = 10^{6}$).
%\small Zoom in on Fig.~\ref{fig:dynam} (d). 
}
\label{fig:scalinglaw}
\end{figure*}

\subsection{The general case}

Computing the integral in terms of the multiplier of the fixed points right after bifurcation, allows us to study the scaling-law for  more general families of maps with a saddle node bifurcation. In the general case, we consider a family of maps $H_{\eps}(x)=H_{\eps}(x,\eps)$ which is real analytic in $x$ and $\eps$ (and hence extendable to a complex analytic one), defined in a neighbourhood of $x=0$, and  given by
{\small 
\begin{equation}\label{eq:H}
H_{\eps}(x)=a\cdot\eps^n +\mathcal{O}(\eps^{n+1})+(1+\mathcal{O}(\eps^{m}))x+(c+\mathcal{O}(\eps))x^2+\mathcal{O}(x^3),
\end{equation}  
}
where $a,c\neq 0$, $n<2m$, $n$ is odd and the higher order terms $\mathcal{O}(x^3)$ may also depend on $\eps$. We assume that all coefficients are real so that $H_{\eps}(x)$ sends the real line to the real line. Since $H_\eps$ is real analytic, there exists $r>0$ such that the Taylor series of $H_{\eps}(x)$ with respect to $x$ has a radius of convergence at least $r$ if $|\eps|$ is small enough. For $\eps=0$, the map $H_0$ has $x=0$ as parabolic fixed point. For simplicity and without loss of generality, we assume that $c<0$ so that $x=0$ is attracting from the right and repelling from the left as was the case for the map $F_0$ (compare Fig.~\ref{fig:epsesquema}). 
Likewise,  we assume that $a<0$ so that the configuration of the fixed points for $\eps\neq 0$ is the same than for $F_{\eps}$. The other cases are derived analogously. 
As it was the case for the family $F_{\eps}$, there is a saddle--node bifurcation which takes place for $\eps=0$.
 
\begin{prop}\label{prop:fixos}
If $|\eps|$ is small, $\eps\neq 0$, then the maps $H_{\eps}$ have exactly two fixed points near $x=0$ given by
$$
x_{\pm}^H(\eps)=\pm\sqrt{\frac{-a}{c}\eps^{n}}\;+\mathcal{O}(\eps^{\frac{n}{2}+t}),
$$
where $t\in \mathbb{R}^+$. Moreover, $t\geq 1/2$ if $n=1$.

 In particular, if $a,c<0$ and $\eps<0$, then there are two real fixed  given by
$$
x_{\pm}^H(\eps)=\pm\sqrt{\frac{a}{c}}\;|\eps|^{\frac{n}{2}}+\mathcal{O}(\eps^{\frac{n}{2}+t})
$$
such that $x_{-}^H(\eps)$ is repelling and $x_{+}^H(\eps)$ is attracting. On the other hand, if $a,c<0$ and $\eps>0$, the two fixed points lie in $\mathbb{C}\setminus \mathbb{R}$ and are given by 
\[
x_{\pm}^H(\eps)=\pm i\sqrt{\frac{a}{c}}\;\eps^{\frac{n}{2}}+\mathcal{O}(\eps^{\frac{n}{2}+t}).
\]
\end{prop}
\begin{proof}
	If $\eps=0$ we have
\[
H_{0}(x)=x+cx^2+\mathcal{O}(x^3),
\]
so $x=0$ is a double fixed point of $H_{0}(x)$ (it is a double solution of $H_{0}(x)=x$). Since $H_{\eps}(x)$ is analytic in $x$ and $\eps$ and non-constant,  it follows from Rouche's Theorem (see, e.g., \cite{Ru}) that for $\eps$ small enough  there are exactly two fixed points of $H_{\eps}(x)$ in a small (complex) neighborhood of $0$,  which converge  to $x=0$ as $\eps\to 0$. We look for candidates of the form $y_{\eps}=\kappa\eps^s+\mathcal{O}(\eps^{s+t})$, where $t,s\in \mathbb{R}^+$, i.e. we want to find $s>0$ and $\kappa\in\mathbb{C}$ such that $y_{\eps}$ is a fixed point of $H_{\eps}(x)$ as $\eps$ tends to 0. Fixed points of $H_{\eps}(x)$ are solutions of $H_{\eps}(x)=x$, which can be written as
\[
0=a\cdot\eps^n+\mathcal{O}(\eps^{n+1})+\mathcal{O}(\eps^{m})x+(c+\mathcal{O}(\eps))x^2+\mathcal{O}(x^3).
\]

Replacing $x$ by $y_{\eps}$ we obtain
%\begin{equation}\label{eq:dom}
%	\arraycolsep=1.4pt\def\arraystretch{2.2}
%	\begin{array}{rl}
%		0=&a\cdot\eps^n+\mathcal{O}(\eps^{n+1})+\mathcal{O}(\eps^{m})(\kappa\eps^s+\mathcal{O}(\eps^{s+t}))
%		+(c+\mathcal{O}(\eps))(\kappa\eps^s+\mathcal{O}(\eps^{s+t}))^2+\mathcal{O}(\eps^{3s})\\
%		=&a\cdot\eps^n+\mathcal{O}(\eps^{n+1})+\mathcal{O}(\eps^{m+s})+\mathcal{O}(\eps^{m+s+t})
%	+c\cdot\kappa^2\cdot\eps^{2s}+\mathcal{O}(\eps^{q})+\mathcal{O}(\eps^{3s}),
%	\end{array}
%\end{equation}
\begin{eqnarray*}\label{eq:dom}
0&=&a\cdot\eps^n+\mathcal{O}(\eps^{n+1})+\mathcal{O}(\eps^{m})(\kappa\eps^s+\mathcal{O}(\eps^{s+t}))
\\&+&(c+\mathcal{O}(\eps))(\kappa\eps^s+\mathcal{O}(\eps^{s+t}))^2+\mathcal{O}(\eps^{3s})
\\&=&a\cdot\eps^n+\mathcal{O}(\eps^{n+1})+\mathcal{O}(\eps^{m+s})+\mathcal{O}(\eps^{m+s+t})
\\&+&c\cdot\kappa^2\cdot\eps^{2s}+\mathcal{O}(\eps^{q})+\mathcal{O}(\eps^{3s}),
\end{eqnarray*}
where $q=\min(2s+1, 2s+t)$. If $y_{\eps}$ is a fixed point, there must be two terms with equal rate of decrease to zero, the slowest of all,  and which cancel out.  These two terms must be among $\mathcal{O}(\eps^n)$, $\mathcal{O}(\eps^{m+s})$ or $\mathcal{O}(\eps^{2s})$. On a closer look however, we see that they need to be $\mathcal{O}(\eps^n)$ and $\mathcal{O}(\eps^{2s})$. Indeed, since $n<2m$, if we had $n=m+s$ then $m>s$ and the term  $\mathcal{O}(\eps^{2s})$ would stand alone as the one with slowest decrease. Likewise, if $m+s=2s$ we would get $m=s$ and the term $\mathcal{O}(\eps^{n})$ would have no pair. Hence, the terms $a\cdot\eps^n$ and $c\cdot\kappa^2\cdot\eps^{2s}$ have to cancel out, which  implies that $s=n/2$ and $\kappa=\pm\sqrt{-a/c}$. The constant $t>0$ can  be found using similar arguments. In particular, it is easy to check that $t\geq 1/2$ if $n=1$. As a result we obtain  
\[
x_{\pm}^H(\eps)=\pm\sqrt{\frac{-a}{c}\eps^{n}}\;+\mathcal{O}(\eps^{\frac{n}{2}+t}).
\]
%Notice that if there were another fixed point $\tilde{y}_{\eps}$ converging to zero any faster or slower than $\pm\sqrt{\frac{-a}{c}\eps^{n}}$, then some dominant term in \eqref{eq:dom} could not be cancelled out, so all fixed points converging to zero have to be of the form $x_{\pm}^H(\eps)$. {\g [No s\'e si aquesta frase anterior cal. A m\'i em confon.]} 
%
To finish the proof we show that both $x_{+}^H(\eps)$ and $x_{-}^H(\eps)$ are fixed points of  $H_{\eps}$.
We assume that $a,c<0$, but similar arguments can be done for the other cases. Since $c<0$, for $\eps=0$ the point $x=0$ is attracting from the right and repelling from the left: there are $z_-<0$ and $z_+>0$ such that $H_0(x)<x$ for all $x\in[z_-,0)\cup(0,z_+]$ (compare Figure~\ref{fig:dynam}~(b)). It follows that if $\eps<0$, $|\eps|$ small enough, then $H_{\eps}(z_-)<z_-$, $H_{\eps}(0)>0$ and $H_{\eps}(z_+)<z_+$ (here we are using that $a<0$ and $n$ is odd). Therefore, there is a fixed point $x_{-}^H(\eps)\in(z_-,0)$ which is repelling in $\mathbb{R}$ and a fixed point $x_{+}^H(\eps)\in(0, z_+)$ which is attracting in $\mathbb{R}$ (compare Figure~\ref{fig:dynam}~(a)).  It can be shown that their multipliers satisfy $\lambda(x_{-}^H(\eps))>1$ and $\lambda(x_{+}^H(\eps))<1$. 
 
 If $\eps\gtrsim 0$, we know that all candidates to fixed points have the form
 \[
 x_{\pm}^H(\eps)=\pm i\sqrt{\frac{a}{c}}\;\eps^{\frac{n}{2}}+\mathcal{O}(\eps^{\frac{n}{2}+t}).
 \]
 In particular, the fixed points of $H_{\eps}(x)$ near $x=0$ lie in $\mathbb{C}\setminus\mathbb{R}$. We will now use the Schwartz Reflection Principle, which establishes that if an analytic $H$ map leaves the real line invariant then $\overline{H(z)}=H(\overline{z})$ for all $z\in\rm{dom}(H)$, where $\overline{z}$ denotes the complex conjugate. Since $\overline{x_{\pm}^H(\eps)}=x_{\mp}^H(\eps)$, it follows that if $x_{+}^H(\eps)$ is a fixed point of $H_{\eps}(x)$, then so is $x_{-}^H(\eps)$. Hence, we cannot have two fixed points of the form $x_{+}^H(\eps)$ or the form $x_{-}^H(\eps)$. This finishes the proof. 
\end{proof}

If $a,c<0$ and $\eps>0$, the two complex fixed points $x_{\pm}^H(\eps)$ are repelling. Indeed,
the derivative $\partial H_{\eps}(x)/\partial x= H'_{\eps}(x)$ of the map $H_{\eps}$ is given by
\[
H'_{\eps}(x)=1+\mathcal{O}(\eps^{m})+(-2|c|+\mathcal{O}(\eps))x+\mathcal{O}(x^2).
\]
We obtain that the multipliers $\lambda(x_{\pm}^H(\eps))=H_{\eps}'(x_{\pm}^H(\eps))$  for $\eps>0$ are given by
\[
\lambda(x_{\pm}^H(\eps))=1\mp i \;2\sqrt{a \; c}\; \eps^{\frac{n}{2}}+\mathcal{O}(\eps^{\frac{n}{2}+\alpha}),
\]
where $\alpha= \mbox{min}\{t,m-n/2,n/2,1\}$. It follows easily that $|\lambda(x_{\pm}^H(\eps))|>1$, so  $x_{\pm}^H(\eps)$ are repelling.

 Similarly to what we did in Theorem~\ref{teor:scaling0}, we now compute the number of iterates required to pass near the former parabolic fixed point $x=0$ for $\eps>0$. 
 \begin{teor}\label{teor:scalinggeneral}
	Let $\delta>0$ small and let $\eps>0$. The number of iterates $N_{\eps}^H$ required to go from $x=\delta$ to $x=-\delta$ under $H_{\eps}$ is given by

\begin{eqnarray*}
	N_{\eps}^H \displaystyle &\approx &
	\frac{2\pi}{\rm{Im}(\lambda( x_-^H(\eps))}+\textcolor{black}{K_1+\mathcal{O}(\eps^{\alpha-n/2})} \\
	&=&\frac{\pi}{\sqrt{a\; c} \;\eps^{\frac{n}{2}}}+\textcolor{black}{K_2+\mathcal{O}(\eps^{\alpha-n/2})}, 
	\end{eqnarray*}
  %{\r [Cal parlar si aix\`o es pot considerar una power law o no. Certament dona un ordre dominant de converg\`encia a infinit. }
where $K_1,K_2\in\mathbb{R}$ and $\alpha >0$. In particular, if $n=1$, then $\alpha\geq n/2$ and hence 
\[
N_{\eps}^H \displaystyle\approx 
	\frac{\pi}{\sqrt{a\; c} \;\eps^\frac12}+\mathcal{O}(1).
\]
 \end{teor}

 \proof
 
 Let $\eta$ be a simple closed curve which intersects the real line exactly in the segment $[-\delta, \delta]$, surrounds the fixed point $x_-^H(\eps)$ (compare Fig.~\ref{fig:eta}), and does not surround any other fixed point of $H_{\eps}$. Let $\eta':=\eta\setminus [-\delta, \delta]$ and let $I_{\eps}^{H,\delta}:=\int_{\eta'}\frac{dz}{H_{\eps}(z)-z}$. 
 Using the same arguments as in Theorem~\ref{teor:scaling0}, we conclude that the number of iterates $N_{\eps}^H$ required to go from $x=\delta$ to $x=-\delta$, where $\delta>0$ is small enough, grows like

\begin{eqnarray*}
N_{\eps}^H \vspace{0.2cm}\displaystyle &\approx &
\frac{2\,\pi\, i }{\lambda(x_{-}^H(\eps))-1}-I_{\eps}^{H,\delta} =\frac{2\pi}{\rm{Im}(\lambda( x_-^H(\eps))} \\ &+&K+\mathcal{O}(\eps^{\alpha-n/2})= \frac{\pi}{\sqrt{a\; c} \;\eps^{1/2}}+K+\mathcal{O}(\eps^{\alpha-n/2}), 
\end{eqnarray*}
 where $K\in\mathbb{R}$. Notice that $I_{\eps}^{H,\delta}=I_0^{H,\delta}+\mathcal{O}(\eps)$. Finally observe that  if $n=1$ we have that $m\geq 1$ and $t\geq 1/2$ (see Prop.\ref{prop:fixos}), hence $\alpha-\frac{1}{2}\geq 0$.  \endproof
 
 \begin{rem}
  Theorem~\ref{teor:scalinggeneral} can be stated analogously  if $a>0$ or $c>0$ (or both). We only state it for $a<0$ and $c<0$ for the sake of simplicity. However, there are two aspects to consider when dealing with $a>0$ or $c>0$. First, depending of the signs of $a$ and $c$, the fixed points $x_{\pm}^H(\eps)$ exit the real line either for positive or negative $\eps$. This has to be taken into account when dealing which each configuration. Second,  if $c>0$ there is a small difference in the arguments. In that case, the parabolic point $x=0$ is attracting from the left and repelling from the right under $H_0$. As a consequence, after perturbation we want to estimate the number of iterates required to go from  go from $x=-\delta$ to $x=\delta$. Since the integral of the holomorphic index is done counter-clockwise over the curve $\eta$, in that case it is convenient to consider the holomorphic index at $x_{+}^H(\eps)$ instead of the one at $x_{-}^H(\eps)$ so that $\oint_{\eta}\frac{dz}{H_{\eps}(z)-z}$ contains the path integral $\int_{-\delta}^{\delta} \frac{dx}{H_{\eps}(x)-x}$. As a result, the formula in Theorem~\ref{teor:scalinggeneral} depends on  $\lambda(x_{+}^H(\eps))$ instead of $\lambda(x_{-}^H(\eps))$.
 \end{rem}
 
 \begin{rem}
  Notice that  after centering the family $F_{\eps}$  at the bifurcation point $x_c$, we obtain the family $G_{\eps}$ in Eq.\eqref{eq:Geps}, which is a particular case of $H_{\eps}$ in Eq.~\eqref{eq:H}. Indeed, $G_{\eps}$ corresponds to a family $H_{\eps}$ with $a=-\mu x_c^2$, $c=-\mu x_c$, and $n=m=1$. Using these parameters,  Theorem~\ref{teor:scaling0} for $F_{\eps}$ follows from Theorem~\ref{teor:scalinggeneral}.
\end{rem}

  We would like to point out that the condition that $n<2m$ is essential to guarantee that a non-degenerate s-n bifurcation takes place. Also, the condition $n$ odd is required to guarantee that the two fixed points change their \textrm{behaviour} before and after perturbation. To finish this section we provide examples of  families for which these conditions are not satisfied and there is no non-degenerate s-n bifurcation. The first family is
  $$H_{1,\eps}(x)=\eps^2+(1+2\eps)x+x^2.$$
  This family has no s-n bifurcation since $x=-\eps$ is a permanent parabolic fixed point of multiplier 1. The second family is
   $$H_{2,\eps}(x)=\eps^2+x+x^2.$$
   If $\eps=0$, $x=0$ is a parabolic fixed point of multiplier 1. However, if $\eps\neq 0$  then the map $H_{2,\eps}$ has two different fixed points in $\mathbb{C}\setminus\mathbb{R}$ which are given by
   $$x_{\pm}=\pm i \eps.$$
   Their multipliers are 
   $$H'_{2,\eps}(x_{\pm})=1\pm i\; 2 \; \eps.$$
   Therefore, for $\eps\neq 0$ small both the fixed points are repelling.
   
   The last family we consider is
   $$H_{3,\eps}(x)=\eps^3+(1+\eps)x+x^2.$$
   If $\eps=0$, $x=0$ is a parabolic fixed point of multiplier 1. If $\eps\neq 0$,  $|\eps|$ small, then there are two real fixed points given by
   $$x_{\pm}=\frac{-\eps\pm \; \eps\sqrt{1-4\eps} }{2}= \frac{-\eps\pm \; (\eps-2\eps^2) }{2}+\mathcal{O}(\eps^3).$$
   Their multipliers are
   $$H'_{3,\eps}(x_-)=1-2\eps+\eps^2 +\mathcal{O}(\eps^3)$$
   and $$H'_{3,\eps}(x_+)=1+\eps-\eps^2+\mathcal{O}(\eps^3).$$

Therefore, one of the fixed points is attracting and the other is repelling, depending on the sign of $\eps$.

\vspace{0.5cm}

\section{Concluding remarks}

In this work we have investigated the dynamical mechanism governing the delays occurring right after a saddle-node (s-n) bifurcation. To do so we have complexified a biological model describing the dynamics of facilitation among individuals of the same species under habitat destruction~\cite{Sardanyes2019}, which has been previously discretised using a single-step Euler approach. Despite the fact that previous research succeeded in obtaining the well-known inverse square-root scaling law using complex analysis~\cite{Fontich2008,DJMS}, the fundamental dynamical mechanism occurring at the complex phase space and responsible for the extremely long delays tangible in the real phase space was still unknown. 

We have here addressed this question by providing a thorough analysis of the complex (as opposed to real) dynamics occurring right after the s-n bifurcation. The complexification of the resulting map has allowed us to identify that the two fixed points undergoing the s-n bifurcation become symmetric unstable spirals in the complex phase space. Right after the s-n bifurcation, both spirals appear to be extremely close to the real line (forming a very narrow channel), where delays take place and the ghost slows down the orbits until the origin (specie's extinction) is achieved. We have studied how the width of the channel in its imaginary dimension affects transients, showing that the inverse square-root scaling law obeys to a linear widening of this channel. Moreover, by determining stability properties of the spirals, we have been able to obtain a simple relation between the multipliers of these fixed points and the scaling law found for passage times post-bifurcation, given by the well-known inverse square-root scaling law~\cite{Strogatz2000}.  Finally, we have proven the same phenomenon for a more general model given by Eq.~\eqref{eq:H}, which includes other mathematical expressions than polynomials. In this sense, our general model considers the bifurcation of two fixed points allowing for much more terms in $\eps$, as a difference from the general model studied in~\cite{Fontich2008}, which only considered a single term in $\eps$ but allowed bifurcations of $2n$ fixed points.  

As far as we know, our work provides, for the first time, a rigorous analysis of the structure of the complex phase space and its holomorphic dynamics behind ghost phenomena, providing the fundamental dynamical mechanism behind the appearance of the extremely long transients governed by ghosts.
%
%\begin{enumerate}
%\item Complexification: numerical results
%\item Estudi dels jets: qu\`e es pot fer?
%\end{enumerate}

\section*{Declaration of competing interests}
The authors declare that they do not have any financial or nonfinancial conflict of interests.

\section*{CRediT authorship contribution statement}
{\bf Jordi Canela}: Conceptualization, Formal analysis, Software, Investigation, Supervision, Writing - original draft. {\bf N\'uria Fagella}:  Conceptualization, Formal analysis, Investigation, Supervision, Writing - original draft, Funding acquisition. {\bf Llu\'is Alsed\`a}: Investigation,  Supervision, Writing - original draft, Funding acquisition. {\bf Josep Sardany\'es}: Conceptualization, Software, Investigation, Writing - original draft, Funding acquisition.

\section*{Acknowledgments}
This work is supported by the Spanish State Research Agency, through the Severo Ochoa and Mar\'ia de Maeztu Program for Centers and Units of Excellence in R\&D (CEX2020-001084-M). We thank CERCA Programme/Generalitat de Catalunya for institutional support.
J.C.  was supported by Spanish Ministry of Economy and Competitiveness, 
through the Mar\'ia de Maeztu Programme MDM-2014-0445, 
by BGSMath Banco de Santander Postdoctoral fellowship 2017, by  the project 
UJI-B2019-18 from Universitat Jaume I, and by PID2020-118281GB-C32. N.F. was partially supported by the MCIN/AEI grants 
MTM2017-86795-C3-3-P and PID2020-118281GB-C32 and the Catalan government grants 2017SGR1374 and 
ICREA Acad\`emia 2020. Ll.A. was partially supported by the MCIN/AEI grants
PID2020-118281GB-C31, MTM2017-86795-C3-1-P and  MDM-2014-0445 within the 
Mar\'ia de Maeztu Program. 
J.S. has been partially funded from AEI grant RTI2018-098322-B-I00 and the Ram\'on y Cajal contract RYC-2017-22243. 
%\end{acknowledgments}

%\bibliographystyle{plain}
%\bibliography{bibliografia}

\begin{thebibliography}{99}


\bibitem{Kuznetsov1998}
Y.\ Kuznetsov  
\newblock {\em Elements of Applied Bifurcation Theory},
\newblock Second Edition. Springer 1998.

\bibitem{Strogatz2000}
S.~H.\ {Strogatz} 
\newblock {\em Nonlinear Dynamics and Chaos with applications to
  Physics, Biology, Chemistry, and Engineering}.
\newblock Westview Press 2000.


\bibitem{Nielsen1993}
M.~K.~{Nielsen} and H.~L.~{Schreyer}.
\newblock Bifurcations in elastic-plastic materials.
\newblock {\em Int. J. of Solids Structures}, 30:521--544, 1993.

\bibitem{Kahan1999}
S.~{Kahan} and A.~C.~{Sicardi-Schifino}.
\newblock Homoclinic bifurcations in {Chua}'s circuit.
\newblock {\em Physica A}, 262:144--152, 1999.


\bibitem{Trickey1998}
S.~T.~Trickey and L.~N.~Virgin.
\newblock Bottlenecking phenomenon near a saddle-node remnant in a {Duffing} oscillator.
\newblock {\em Phys. Lett. A}, 248:185--190, 1998.

\bibitem{Ivanchenko2017}
M.~{Ivanchenko}, E.~{Kozinov}, V.~{Volokitin}, A.~{Liniov}, I.~{Meyerov}, and
  S.~{Denisov}.
\newblock Classical bifurcation diagrams by quantum means.
\newblock {\em Annalen der Physik}, 529:1600402, 2017.

\bibitem{Carpenter2011}
S.~R.~{Carpenter et al.}
\newblock Early warnings of regime shifts: A whole ecosystem experiment.
\newblock {\em Science}, 332:1709--1082, 2011.

\bibitem{Rietkerk2004}
M.~Rietkerk, S.~C. Dekker, P.~C. de~Ruiter, and J.~van~de Koppel.
\newblock Self-organized patchiness and catastrophic shifts in ecosystems.
\newblock {\em Science}, 305:1926--1929, 2004.

\bibitem{Staver2011}
A.~C.~Staver, S.~Archibald, and S.~A.~Levin.
\newblock Anticipating critical transitions.
\newblock {\em Science}, 334:230--232, 2011.

\bibitem{Hastings2018}
A.~Hastings, K.~C.~Abbott, K.~Cuddington, T.~Francis, G.~Gellner, Y.C.~Lai, A.~Morozov, S.~Petrovskii, K.~Scranton, and Y.~C.~Zeeman.
\newblock {Transient phenomena in ecology}
\newblock {\em Science} {\bf 361} (6406) eaat6412, 2018.

\bibitem{Morozov2020}
A.~Morozov, K.~Abbott, K.~Cuddington, T.~Francis, G.~Gellner, A.~Hastings, Y.~C.~Lai, S.~Petrovskii, K.~Scranton and M.~L.~Zeeman. 
\newblock {Long transients in ecology: Theory and applications}.
\newblock {\em Phys. Life Rev.} {\bf{32}} 1-40, 2020.

%\bibitem{Rietkerk2004}
%{Rietkerk} M, {Dekker} SC, {de Ruiter} PC, {van de Koppel} J (2004)
%\newblock {Self-organized patchiness and catastrophic shifts in ecosystems}.
%\newblock {\em Science} 305:1926--1929.


\bibitem{May2008}
R.~M.~May and S.~A.~Levin.
\newblock Complex systems: Ecology for bankers.
\newblock {\em Science}, 338:344--348, 2008.

\bibitem{Lade2013}
S.~J.~Lade, A.~Tavoni, S.~A.~Levin, and M.~Schl\"uter.
\newblock Regime shifts in a socio-ecological system.
\newblock {\em Theor. Ecol.}, 6:359--372, 2013.
\bibitem{Fontich2008}
{Fontich} E, {Sardany\'es} J (2008)
\newblock {General scaling law in the saddle-node bifurcation: a complex phase
	space study}.
\newblock {\em J. Phys. A: Math. Theor.} {\bf{41}}, 468-482, 2008.

\bibitem{Gimeno2018}
J.~{Gimeno}, \`A.~{Jorba}, and J.~{Sardany\'es}.
\newblock {On the effect of time lags on a saddle-node remnant in hyperbolic replicators}.
\newblock {\em J. Phys. A: Math. Theor.}, {\bf{51}}:385601, 2018.

\bibitem{Sardanyes2019}
J.~Sardany\'es, J.~Pi\~nero, and R.~Sol\'e.
\newblock Habitat loss-induced tipping points in metapopulations with
facilitation.
\newblock {\em Pop. Ecol.}, 61(4):436--449, 2019.



\bibitem{Murray2002}
J.~D. Murray.
\newblock {\em Mathematical Biology: I. An Introduction}.
\newblock Springer-Verlag, New York, 2002.

\bibitem{Ott2002}
E.~Ott.
\newblock {\em Chaos in Dynamical Systems}.
\newblock Cambridge University Press, Cambridge, 2002.
 
 \bibitem{Hinrichsen2000}
H.~Hinrichsen. Non-equilibrium critical phenomena and phase transitions into
absorbing states. \newblock {\em Adv. Phys.} 49(7):815-958, 2000



\bibitem{Odor2008}
G.~\'Odor. Universality in Non-equilibrium Lattice Systems: Theoretical Foundations (World Scientific, Singapore), 276-572, 2008.


\bibitem{Gil1991}
L.~Gil, G.~Balzer, P.~Coullet, M.~Dubois, and P.~Berge.
\newblock Hopf bifurcation in a broken-parity pattern.
\newblock {\em Phys. Rev. Lett.}, 66:3249--3255, 1991.

\bibitem{Das2007}
M.~Das, A.~Vaziri, A.~Kudrolli, and L.~Mahadevan.
\newblock Curvature condensation and bifurcation in an elastic shell.
\newblock {\em Phys. Rev. Lett.}, 98:014301, 2007.

\bibitem{Gomez2017}
M.~Gomez, D.~E. Moulton, and D.~Vella.
\newblock Critical slowing down in purely elastic 'snap-through' instabilities.
\newblock {\em Nature Phys.}, 13:142--145, 2017.

\bibitem{Maselko1982}
J.~{Maselko}. 
\newblock {Determination of bifurcation in chemical systems. An experimental method}.
\newblock {\em Chem. Phys.} 67:17--26, 1982.


\bibitem{Strizhak1996}
P.~{Strizhak} and M.~{Menzinger}.
\newblock {Slow-passage through a supercritical Hop bifurcation: Time-delayed
	response in the Belousov-Zhabotinsky reaction in a batch reactor}.
\newblock {\em J. Chem. Phys.} 105:10905, 1996.


\bibitem{Dai2012}
L.~{Dai}, D.~{Vorselen},K.~S.~{Korolev}, and J.~{Gore}.
\newblock {Generic indicators of loss of resilience before a tipping point
	leading to population collapse}.
\newblock {\em Science} 336:1175--1177, 2012.





\bibitem{Gu2014}
H.~Gu, B.~Pan, G.~Chen, and L.~Duan.
\newblock Biological experimental demonstration of bifurcations from bursting to spiking predicted by theoretical models.
\newblock {\em Nonlinear Dyn.}, 78:391--407, 2014.

\bibitem{Leonel2016}
E.~D.~Leonel. Defining universality classes for three different local bifurcations. \newblock {\em Commun. Nonlinear Sci. Numer. Simulat.}  {\bf 39}, 520-528, 2016.


\bibitem{Teixeira2015}
R.~M.~N.~Teixeira, D.S.~Rando, F.~C.~Geraldo, R.~N.~Costa-Filho, J.~A.~Oliveira, and E.~D.~Leonel. Convergence towards asymptotic state in 1-D mappings: A scaling investigation \newblock {\em Phys. Lett. A} {\bf{379}}(18-19), 1246-1250, 2015.



\bibitem{Duarte2011}
J.~Duarte, C.~Janu\'ario, N.~Martins, and  J.~Sardany\'es. 
Scaling law in saddle-node bifurcations for one-dimensional maps: a complex variable approach. \newblock {\em Nonlin. Dyn.} {\bf 67}, 541-547, 2012.

\bibitem{Sardanyes2006}
J.~{Sardany\'es} and R.~V.~{Sol\'e}.
\newblock {Ghosts in the origins of life?}.
\newblock {\em Int J. Bif. and Chaos} 16(9):2761--2765, 2006.


\bibitem{Strogatz1989}
S.~H.~Strogatz and R.~M.~Westervelt. Predicted power laws for delayed switching of charge density waves. Phys. Rev. B 40(15): 10501-10508, 1989.


\bibitem{Sardanyes2007}
J.~{Sardany\'es} and R.~V.~{Sol\'e}.
\newblock {The role of cooperation and parasites in non-linear replicator delayed extinctions}.
\newblock {\em Chaos, solitons \& fractals} {\bf{31}}(5), 1279--1296, 2007.


\bibitem{Vidiella2018}
B.~Vidiella, J.~Sardany\'es, R.~Sol\'e.  Exploiting delayed transitions to sustain semiarid ecosystems after catastrophic shifts. \newblock {\em J. Royal Soc Interface} 15:20180083, 2018.

\bibitem{DJMS}
J.~Duarte, C.~Janu\'ario, N.~Martins, and J.~Sardany\'es.
\newblock Scaling law in saddle-node bifurcations for one-dimensional maps: a
complex variable approach.
\newblock {\em Nonlinear Dyn.}, 67, 541--547, 2012.


\bibitem{Szostakiewicz2014}
M.~Szostakiewicz, M.~Urb\'anski, and A.~Zdunik.
\newblock{Stochastics and thermodynamics for equilibrium
	measures of holomorphic endomorphisms on complex
	projective spaces}.
\newblock {\em Monatsh. Math.}, 174: 141-162, 2014.

\bibitem{Coleman1984}
C.~J.~Coleman. 
\newblock{On the use of complex variables in the analysis of flows of an elastic fluid} 
\newblock{\em J. Non-Newtonian Fluid Mech. }15, 227-238, 1984.

\bibitem{Marner2017}
F.~Marner, P.~H.~Gaskell, and M.~Scholle.
\newblock{A complex-valued first integral of Navier-Stokes equations: Unsteady Couette flow in a corrugated channel system}
\newblock{\em J. Math. Phys.} 58, 043102, 2017.

\bibitem{Dirac1937}
P.~A.~M.~Dirac.
\newblock{Complex variables in quantum mechanics}
\newblock{\em Proc. R. Soc. Lonf. A} 160, 48-59, 1937.

\bibitem{Cacuci2010}
D.~G.~Cacuci and M.~Ionescu-Bujor
\newblock{Ed. Cacuci, Dan Gabriel},
\newblock{\em Mathematics for Nuclear Engineering. Handbook of Nuclear Engineering}
\newblock {Springer US, Boston, MA, 2010.}
%pages="643--749",

\bibitem{Poozesh2016}
A.~Poozesh and M.~Mirzaei.
\newblock{Flow Simulation Around Cambered Airfoil by Using
	Conformal Mapping and Intermediate Domain in Lattice
	Boltzmann Method}
\newblock{\em J. Stat. Phys.} 166, 354-367, 2016.

\bibitem{Bird2007}
J.~Bird.
\newblock{\em Application of complex numbers to series a.c. circuits}
\newblock{Book Electrical Circuit Theory and Technology. 3rd Edition, 2007}

\bibitem{GS}
J.~Graczyk and G.~\'Swiatek. 
\newblock Generic hyperbolicity in the logistic family.
\newblock {\em Ann. of Math.(2)}, 146(1):1-52, 1997.

\bibitem{Lyu}
M.~Lyubich.
\newblock Dynamics of quadratic polynomials. I, II.
\newblock {\em Acta  Math.}, 178(2), 247-297, 1997.

\bibitem{Fontich2010}
E.~{Fontich} and J.~{Sardany\'es}.
\newblock {On the metapopulation dynamics of autocatalysis: Extinction transients related to ghosts}.
\newblock {\em Int. J. Bifurc. Chaos} {\bf{20}}, 1261-1268468-482, 2010.

\bibitem{Duarte2012}
J.~Duarte, C.~Janu\'ario, N.~Martins, and J.~Sardany\'es. 
On chaos, transient chaos and ghosts in single population models with Allee effects. \newblock {\em Nonlin. Analysis: Real World Appl.} {\bf 13}, 1647-1661, 2012.

\bibitem{Sardanyes2006a}
J.~{Sardany\'es} and R.~V.~{Sol\'e}.
\newblock {Bifurcations and phase transitions in spatially-extended two-member hypercycles}.
\newblock {\em j. theor. Biol.} {\bf{243}}, 468--482, 2006.

\bibitem{Mi}
J.~Milnor.
\newblock {\em Dynamics in one complex variable}, volume 160 of {\em Annals of
	Mathematics Studies}.
\newblock Princeton University Press, Princeton, NJ, third edition, 2006.


\bibitem{Oud}
R.~Oudkerk.
\newblock The Parabolic implosion for $f_0(z)=z+z^{\nu+1}+\mathcal{O}(z^{\nu+2})$.
\newblock {\em PhD Thesis}, 1999.


\bibitem{Ru}
W.~Rudin 
\newblock {\em Real and complex analysis},
\newblock Third Edition. McGraw-Hill Book Co., 1987.


%\bibitem{Scheffer2001}
%M.~{Scheffer}, et al.
%\newblock {Catastrophic shifts in ecosystems}.
%\newblock {\em Nature} 413:591--596, 2001.

%\bibitem{Strizhak1996}
%P.~Strizhak and M.~Menzinger.
%\newblock Slow-passage through a supercritical {Hopf} bifurcation: Time-delayed
 % response in the {Belousov-Zhabotinsky} reaction in a batch reactor.
%\newblock {\em J. Chem. Phys.}, 105:10905, 1996.

%\bibitem{Dai2012}
%L.~Dai, D.~Vorselen, K.~S. Korolev, and J.~Gore.
%\newblock Generic indicators of loss of resilience before a tipping point
 % leading to population collapse.
%\newblock {\em Science}, 336:1175--1177, 2012.

%\bibitem{Gu2014}
%H.~{Gu}, B.~{Pan}, G.~{Chen}, and L.~{Duan}.
%\newblock {Biological experimental demonstration of bifurcations from bursting
%  to spiking predicted by theoretical models}.
%\newblock {\em Nonlinear Dyn.}, {\bf{78}}:391--407, (2014).


%\bibitem{Scheffer2003}
%M.~{Scheffer} and S.~ R.~{Carpenter}.
%\newblock {Catastrophic regime shifts in ecosystems: linking theory to obervation}.
%\newblock {\em Trends Ecol Evol} 18(12):648--656, 2003.

%\bibitem{Su}
%D.~Sullivan.
%\newblock Quasiconformal homeomorphisms and dynamics. {I}. {S}olution of the
%{F}atou-{J}ulia problem on wandering domains.
%\newblock {\em Ann. of Math. (2)}, 122(3):401--418, 1985.







\end{thebibliography}

\end{document}